\documentclass[final, 3p, times]{elsarticle}
\usepackage{graphicx}
\usepackage{xcolor}
\usepackage{colortbl}
\usepackage{epstopdf}
\usepackage{amssymb}
\usepackage{amsthm}
\usepackage{amsmath}
\usepackage{color}
\usepackage{dsfont}
\usepackage{mathtools}
\usepackage{extarrows}
\usepackage{bm}
\usepackage{caption}
\usepackage{float}
\usepackage{verbatim}
\usepackage{epsfig}
\usepackage{multirow}
\usepackage{subfigure}
\usepackage{booktabs}
\usepackage{array}
\usepackage{threeparttable}
\usepackage{enumitem}
\usepackage{algorithm}
\usepackage{algorithmic}
\usepackage{hyperref}
\usepackage{cleveref}
\crefname{algorithm}{Algorithm}{Algorithms}
\crefname{table}{Table}{Tables}
\usepackage[section]{placeins}

\numberwithin{equation}{section}%
\newtheorem{lemma}{Lemma}[section]
\crefname{lemma}{Lemma}{Lemma}

\newtheorem{theorem}{Theorem}[section]
\crefname{theorem}{Theorem}{Theorem}

\newtheorem{definition}{Definition}[section]
\crefname{definition}{Definition}{Definition}

\newtheorem{remark}{Remark}[section]
\newtheorem{example}{Example}

\crefname{example}{Example}{Examples}
\crefname{table}{Table}{Tables}
\crefname{figure}{Figure}{Figures}
\biboptions{compress}

\begin{document}
\begin{frontmatter}
\title{Higher-Order Finite Difference Methods for the Tempered Fractional Laplacian}

\tnotetext[label1]{The research of Dongling Wang is supported in part by the National Natural Science Foundation of China under grants 12271463,
the 111 Project (No.D23017) and Program for Science and Technology Innovative Research Team in Higher Educational Institutions of Hunan Province of China.
The work of Mingyi Wang is supported by Hunan Provincial Innovation Foundation For Postgraduate (No. CX20240602). \\ Declarations of interest: none.}

\author[XTU]{Mingyi Wang} 
\ead{202331510144@smail.xtu.edu.cn}
\author[XTU]{Dongling Wang\corref{mycorrespondingauthor}}
\ead{wdymath@xtu.edu.cn, ORCID 0000-0001-8509-2837}
\cortext[mycorrespondingauthor]{Corresponding author.}

\address[XTU]{Hunan Research Center of the Basic Discipline Fundamental Algorithmic Theory and Novel Computational Methods, National Center for Applied Mathematics in Hunan, School of Mathematics and Computational Science, Xiangtan University, Xiangtan 411105, Hunan, China }

\begin{abstract}
This paper presents a general framework of high-order finite difference (HFD) schemes for the tempered fractional Laplacian (TFL) based on new generating functions obtained from the discrete symbols. Specifically, for sufficiently smooth functions, the resulting discretizations achieve high-order convergence with orders
$p=4, 6, 8$. The discrete operators lead to Toeplitz stiffness matrices, allowing efficient matrix-vector multiplications via fast algorithms. Building on these approximations, HFD methods are formulated for solving TFL equations, and their stability and convergence are rigorously analyzed. Numerical simulations confirm the effectiveness of the proposed methods, showing excellent agreement with the theoretical predictions.
\end{abstract}

\begin{keyword}
 Tempered fractional Laplacian equations, generating functions, high-order finite difference methods, stability, convergence.
\end{keyword}
\end{frontmatter}

\section{Introduction}
	\setcounter{equation}{0}

In this paper, we develop high-order finite difference (HFD) methods for the following tempered fractional Laplacian (TFL) equations:
\begin{align}\label{eq:1.1}
(-\Delta)_{\lambda}^{\frac{\alpha}{2}}u(x)-\sigma \Delta u(x)+\nu u(x) &= f(x),\ \ x\in\Omega \subset{\mathbb{R}^d},\ \  \\ \label{eq:1.2}
 u(x) &= 0,\ \ x\in \Omega^c\equiv\mathbb{R}^d\backslash\Omega,
\end{align}
where $\alpha\in(0,1)\cup(1,2)$, the tempering parameter $\lambda\geq0$, the diffusion coefficient $\sigma\geq0$, and the reaction coefficient $\nu\geq0$; $\Omega$ is a bounded open domain, $d$ is the spatial dimension,  $\alpha$ stands for the fractional power, $f$ is the source term, and the TFL $(-\Delta)_{\lambda}^{\frac{\alpha}{2}}$ is defined via the Cauchy principal value ($\text{P.V.}$) hypersingular integral
\ \cite{CDH2025,DZ2019JSC,YDN2020},
\begin{align}\label{eq:3}
(-\Delta)_{\lambda}^{\frac{\alpha}{2}}u(x):=\frac{1}{|\Gamma(-\alpha)|}\text{P.V.}\int_{\mathbb{R}^d}\frac{u(x)-u(y)}{e^{\lambda|x-y|}|x-y|^{d+\alpha}}dy,
\end{align}
where $\Gamma(\cdot)$ is the Gamma function.

In L\'evy flights characterized by long-range jumps, particles can leap from arbitrarily remote locations directly into or out of a bounded domain. Consequently, the dynamics inside the domain inherently depend on information from the whole space. To address this nonlocality, generalized Dirichlet boundary conditions are introduced to prescribe the solution over the entire complementary region. These conditions ensure physical consistency and mathematical well-posedness by incorporating nonlocal interactions, thereby enabling accurate and meaningful predictions for anomalous diffusion in bounded domains.

This paper focuses on a TFL-equivalent pseudo-differential operator \cite{CDH2025, DLTZ2017} on the whole space $\mathbb{R}^d$
\begin{align}\label{eq:4}
    (-\Delta)_{\lambda}^{\frac{\alpha}{2}}u(x)=\mathcal{F}^{-1}\left[S^{\alpha}_{\lambda}\left(\xi\right)\mathcal{F}\left[u\right]\left(\xi\right)\right]\left(x\right),\ \ \alpha\in(0,1)\cup(1,2),
\end{align}
where $S^{\alpha}_{\lambda}(\xi)$ is the frequency-domain symbol of the TFL, which is given by
\begin{align}\label{eq:1.5}
S^{\alpha}_{\lambda}(\xi) :&= (-1)^{\lfloor \alpha\rfloor}\int_{\|\theta\|=1}((\lambda+i \xi\cdot\theta)^\alpha-\lambda^\alpha)d\theta \\
&=(-1)^{\lfloor \alpha\rfloor}\cdot
\begin{cases}
((\lambda+i \xi_1)^\alpha +(\lambda-i \xi_1)^\alpha-2\lambda^\alpha),\ \  d = 1,\nonumber\\
\left(\int_{0}^{2\pi}(\lambda+i \xi_1\cos\theta_1+i\xi_2\sin\theta_1)^\alpha d\theta_1-2\pi\lambda^\alpha\right),\ \ d =2,\nonumber\\
\left(\int_{0}^{2\pi}\int_{0}^{\pi}(\lambda+i \xi_1\cos\theta_2\sin\theta_1+i\xi_2\sin\theta_2 \sin \theta_1+i\xi_3\cos\theta_1)^\alpha \sin \theta_1d\theta_1d \theta_2-4\pi\lambda^\alpha\right),\ d =3.
\end{cases}
\end{align}
And $\mathcal{F}$ and $\mathcal{F}^{-1}$ represent the Fourier transform and inverse transform:
\[
\hat{u}(\xi) = \mathcal{F}[ u](\xi) = \int_{\mathbb{R}^d} u(x)e^{-i\xi\cdot x}dx, \quad \mathcal{F}^{-1}[\hat{u}](x)=\frac{1}{(2\pi)^d}\int_{\mathbb{R}^d}\hat{u}(\xi)e^{i\xi\cdot x}d\xi.
\]
When $\lambda=0$, the TFL reduces to the standard fractional Laplacian up to a constant such that
\[
(-\Delta)_{\lambda}^{\frac{\alpha}{2}}u(x)=\mathcal{F}^{-1}\left[ C_{d,\alpha}|\xi|^{\alpha}\mathcal{F}[u](\xi)\right](x),
\quad C_{d,\alpha}=\frac{\pi^{{d}/{2}}|\Gamma(-{\alpha}/{2})|}{2^{\alpha-1}\alpha\Gamma((d+\alpha)/{2})|\Gamma(-\alpha)|}.
\]

As the infinitesimal generator of isotropic Brownian motion, the negative Laplacian $-\Delta$ serves as the fundamental local operator of the limit process of normal diffusion. The fractional Laplacian $(-\Delta)^{\frac{\alpha}{2}}$ is the infinitesimal generator of a symmetric $\alpha$-stable isotropic L$\acute{\text{e}}$vy process, denoting a fundamental nonlocal operator of the limit process of anomalous diffusion, and representing a nonlocal generalization of the Laplacian \cite{A2009,DG2013,DWZ2018,P2016}. In some practical models, in order to avoid the divergence of the second and higher-order moments caused by excessively long jumps in the L\'evy process, an exponentially decaying factor with a sufficiently small $\lambda$ is incorporated into the fractional Laplacian to mollify the power law of isotropic distances, thus generating a new nonlocal operator, namely, the tempered fractional Laplacian $(-\Delta)_{\lambda}^{\frac{\alpha}{2}}$. This operator can better capture the transition phenomena from early-stage anomalous diffusion to late-stage normal diffusion, representing a further generalization of the fractional Laplacian \cite{DZ2019,P2016}. For the corresponding mathematical derivations, readers can refer to \cite{DLTZ2017,DWZ2018}.

The concept of fractional derivatives dates back to a question posed to Leibniz in the $17$-th century. Since then, the theory has evolved substantially and is now widely used to model phenomena in physics, finance, biology, hydrology, and other fields \cite{SMC2015}. A notable application arises from the observed coexistence and transition of diffusion across many disciplines, including biology \cite{JHMJM2013}, turbulence \cite{DL1998}, geophysics \cite{MZB2008,ZMP2012}, and finance \cite{CGMY2002,CGMY2003}. To characterize such diffusion transition behavior, researchers have employed various approaches and confirmed that the tempered fractional Laplacian provides an effective mathematical model; see, e.g., \cite{BM2010,CD2007,CM2011,DWZ2018}. Consequently, TFL equations offer significant advantages both in mathematical theory and in practical applications.

In recent years, the broad applications and absence of analytical solutions for the TFL equations, coupled with the numerical challenges arising from the nonlocality and singularity of the TFL, have motivated numerous scholars to develop efficient and simple numerical methods. At present, finite element methods \cite{AB2017,ABB2017,DG2013,SL2024}, finite difference methods \cite{DZ2019,HW2022,HZD2021,HO2014,HO2016,MY2020,YCYL2025} and spectral methods \cite{MS2017,TWYZ2020,ZZ2024} have been extensively developed for the standard fractional Laplacian. Owing to the presence of the tempering parameter $\lambda$ within the kernel function of the TFL, its equivalent formulations \cite{K2017} and mathematical properties differ from those of the standard fractional Laplacian. As a result, some numerical methods developed for the fractional Laplacian equations are ineffective for solving the TFL equations, and it is of significant importance to establish an efficient and simple numerical solver for the TFL equations.

Several numerical methods have been proposed to address the computational challenges of TFL equations. Zhang et al. introduced a Riesz basis Galerkin method for TFL equations with Dirichlet boundary conditions, establishing the well-posedness of the variational formulation and proving the convergence of the Galerkin approximation \cite{ZDK2018}. They later developed a one-dimensional finite difference scheme, showing that the convergence rates depend on the solution regularity over bounded domains \cite{ZDF2018}. For high-dimensional problems, Yan et al. \cite{YDN2020} used linear and bilinear interpolations to approximate the nonsingular parts of the one- and two-dimensional TFL, respectively, while applying quadratic and biquadratic interpolations for the singular parts, and provided estimates for consistency and numerical errors. Extending the one-dimensional finite difference scheme from \cite{ZDF2018} to two dimensions, a method combining the weighted trapezoidal rule with bilinear interpolation was proposed, along with corresponding convergence analysis. To further tackle high-dimensional TFL equations with improved accuracy, Duo and Zhang \cite{DZ2019} developed a finite difference scheme for the three-dimensional TFL that achieves second-order truncation accuracy under relatively low smoothness assumptions, though stability and convergence of the numerical solutions remain unproven. More recently, Cui and Hao \cite{CDH2025} provided an equivalent pseudo-differential operator definition for the hypersingular TFL, constructed a simple second-order finite difference scheme, and rigorously analyzed its stability and convergence.

To the best of our knowledge, existing numerical methods for TFL equations achieve at most second-order convergence even for sufficiently smooth solutions. To improve the convergence rate, we develop a general framework of high-order finite difference (HFD) schemes for the tempered fractional Laplacian $(-\Delta)_{\lambda}^{\frac{\alpha}{2}}$ and provide a rigorous stability and convergence analysis for the proposed methods. 
Our main contributions can be summarized as:

\begin{enumerate}[label=(\arabic*), leftmargin=3.3em]
\item By utilizing the generating functions of HFD schemes for the classical Laplacian and the semi-discrete Fourier transform, we approximate the continuous Fourier transform of the pseudo-differential operator with the semi-discrete transform. This leads to the discrete TFL operator
\begin{align}
(-\Delta_h)_{\lambda}^{\frac{\alpha}{2}}u(x):= \frac{1}{h^{\alpha}}\sum_{j\in \mathbb{Z}^d} a_{j}^{(\alpha,h\lambda)}u(x+jh),
\end{align}
where the weights $a_{j}^{(\alpha,h\lambda)}$ produce a multilevel Toeplitz stiffness matrix, enabling efficient computation via the fast Fourier transform.

\item For the discrete TFL schemes, we prove $l^{\infty}$-norm convergence rates of order $O(h^{p})$ for $p=4,6,8$, under sufficient smoothness of the exact solution. For exact solutions $u\in\mathcal{B}^s(\mathbb{R}^d)$ with $s>0$, we establish the stability of the HFD methods for TFL equations and derive $l^2$-norm convergence rates of order $O\big(h^{\min\{s-\alpha, p\}}\big)$. The analysis relies on the discrete TFL framework and the boundedness and convergence properties of the continuous and discrete symbols of the TFL.
\end{enumerate}

This paper is organized as follows. In Section 2, we first introduce the generating functions of the finite difference schemes of orders $p=4,6,8$ for the classical Laplacian. Using these functions, we then design high-order finite difference (HFD) schemes for the tempered fractional Laplacian and analyze the consistency errors of the discrete TFL operator. Section 3 builds on the approximations for the TFL and the classical Laplacian to develop HFD methods for the TFL equations, and provides a detailed stability and convergence analysis. In Section 4, several numerical examples are presented to illustrate the effectiveness and feasibility of the proposed methods. Finally, Section 5 offers a summary and concluding remarks.

\section{The HFD schemes for the TFL}
This section constructs the HFD schemes for the TFL, derived from the frequency-domain symbols of HFD schemes for the classical Laplacian. It then provides a rigorous convergence proof for the discrete TFL, and detailed implementation procedures for the proposed schemes.

\subsection{The approximation for the TFL}
For smooth functions
$u$, when $\alpha=2$ and $\lambda = 0$, the TFL reduces to the classical negative Laplacian. By employing the continuous Fourier transform defined in \eqref{eq:1.5}, it can be derived that
\begin{align}
\mathcal{F}\left[-\Delta u\right](\xi)=\psi(\xi)\mathcal{F}[u](\xi),
\end{align}
where $\psi(\xi):=|\xi|^2$ stands for the frequency-domain symbol of the classical negative Laplacian.
The HFD schemes $-\Delta_ h$ for the classical negative Laplacian can be constructed on a uniform grid via Taylor series expansions, requiring $u\in C^6[x-2h,x+2h]$, $C^8[x-3h,x+3h]$, and $C^{10}[x-4h,x+4h]$, respectively. We have
\begin{align}\label{eq:2.1}
-\Delta u(x)&=-\Delta_hu(x)+O(h^p)\\
&:=\frac{1}{h^2}\begin{cases}
\sum\limits_{l=1}^{d}\left(\frac{1}{12}u\left(x-2he_l\right)-\frac{16}{12}u(x-he_l)+\frac{30}{12}u(x)\right.\nonumber\\
 \ \ \ \ \ \left. -\frac{16}{12}u(x+he_l)+\frac{1}{12}u(x+2he_l)\right)+O(h^4),\ \ p = 4,\\
\sum\limits_{l=1}^{d}\left(-\frac{1}{90}u(x-3he_l)+\frac{3}{20}u(x-2he_l)-\frac{3}{2}u(x-he_l)+\frac{49}{18}u(x)\right.\nonumber\\
 \ \ \ \ \ \left.-\frac{3}{2}u(x+he_l)+\frac{3}{20}u(x+2he_l)-\frac{1}{90}u(x+3he_l)\right)+O(h^6),\ \ p = 6,\\
 \sum\limits_{l=1}^{d}\left(\frac{1}{560}u(x-4he_l)-\frac{8}{315}u(x-3he_l)+\frac{1}{5}u(x-2he_l)-\frac{8}{5}u(x-he_l)+\frac{205}{72}u(x)\right.\nonumber\\
 \ \ \ \ \ \left.-\frac{8}{5}u(x+he_l)+\frac{1}{5}u(x+2he_l)-\frac{8}{315}u(x+3he_l)+\frac{1}{560}u(x+4he_l)\right)+O(h^8),\ \ p = 8,
\end{cases}
\end{align}
where $h$ is the step size of the uniform grid, the grid points are given by $x\in h\mathbb{Z}^d=\{jh\ |\ j=(j_1, \ldots, j_d)\in \mathbb{Z}^d\}$, and $e_l$ denotes the $l$-th canonical basis vector.

To construct the HFD schemes for the TFL in \eqref{eq:4}, it is necessary to introduce the definition of the semi-discrete Fourier transform.
\begin{definition}(Semi-discrete Fourier transform \cite{CDH2025,HZD2021,ZZ2024}). For a grid function $U=\{u(x)\}_{x\in h\mathbb{Z}^d}\in l^2(\mathbb{R}^d)$, the semi-discrete Fourier transform $\mathcal{F}_h$ and inverse transform $\mathcal{F}^{-1}_{h}$ are defined as follows
\begin{align}\label{eq:2.2}
&\check{u}(\xi)=\mathcal{F}_h\left[U\right](\xi)=h^d\sum_{x\in h\mathbb{Z}^d}u(x)e^{-i\xi\cdot x},\ \ \xi\in \left[-\frac{\pi}{h}, \frac{\pi}{h}\right]^d,\\ \label{eq:2.3}
&u(x) = \mathcal{F}^{-1}_{h}\left[\check{u}\right](x)=\frac{1}{(2\pi)^d}\int_{D_h}\check{u}(\xi)e^{i\xi \cdot x}d\xi,
\end{align}
where $D_h := \left[-\frac{\pi}{h},\frac{\pi}{h}\right]^d$. Here, $u(x)$ is merely the value of the function at the uniform grid point $x$ and does not represent the solution to TFL equation.
\end{definition}

In order to present HFD clearly, we first consider the specific construction process of a simple one-dimensional case for $d=1$, and then extend it to high-dimensional cases for $d=2, 3$.

For the case of $d=1$, we apply the semi-discrete Fourier transform $\mathcal{F}_h$ in \eqref{eq:2.2} to the grid functions $(-\Delta_h) U$ of the discrete negative Laplacian in \eqref{eq:2.1} and derive the symbol of this discrete operator. That is
\begin{align}\label{eq:2.4}
\mathcal{F}_h[(-\Delta_h) U](\xi) = \psi_h(\xi )\mathcal{F}_h[U](\xi)= \psi_h(\xi )\check{u}(\xi),
\end{align}
where $\psi_h(\xi)$ is the discrete symbol of the discrete negative Laplacian $-\Delta_h$. We have
\begin{align}\label{eq:2.6}
\psi_h(\xi ):=\frac{1}{h^2}\begin{cases}\left(\frac{1}{6}\cos(2\xi h)-\frac{8}{3}\cos(\xi h)+\frac{5}{2}\right),\ \ p =4,\\
\left(-\frac{1}{45}\cos(3\xi h)+\frac{3}{10}\cos(2\xi h)-3\cos(\xi h)+\frac{49}{18}\right),\ \ p = 6,\\
\left(\frac{1}{280}\cos(4\xi h)-\frac{16}{315}\cos(3\xi h)+\frac{2}{5}\cos(2\xi h)-\frac{16}{5}\cos(\xi h)+\frac{205}{72}\right),\ \ p = 8.
\end{cases}
\end{align}

Let
\begin{align}\label{eq:02.9}
\xi_h:=\begin{cases}\psi_h(\xi )^{\frac{1}{2}},\ \ \xi\geq0,\\
-\psi_h(\xi)^{\frac{1}{2}},\ \ \xi<0.
\end{cases}
\end{align}
When extending to high-dimensional cases, where $\xi = (\xi_1,\ldots,\xi_d)$ with $d=2, 3$,  each component
$\xi_l$ admits two possible forms. Consequently, the definition of $\xi_h$ can be expressed as a combination of the terms
$\psi_h(\xi _l)^{\frac{1}{2}}$ and $-\psi_h(\xi_l )^{\frac{1}{2}}$ for  $l =1,\ldots,d$.

The semi-discrete Fourier transform $\mathcal{F}_h$ and its inverse transform $\mathcal{F}^{-1}_h$ in \eqref{eq:2.2} and \eqref{eq:2.3} approximate the continuous Fourier transform $\mathcal{F}$ and its inverse transform $\mathcal{F}^{-1}$ in \eqref{eq:4}, respectively. Consequently, the HFD schemes for the TFL can be derived directly from \eqref{eq:4} as follows
\begin{align} \label{eq:2.9}(-\Delta_h)_{\lambda}^{\frac{\alpha}{2}}u(x)&=\mathcal{F}_h^{-1}\left[S^{\alpha}_{\lambda,h}(\xi)\mathcal{F}_h\left[U\right](\xi)\right](x),
\end{align}
where
\begin{align}\label{eq:2.8}
S^{\alpha}_{\lambda,h}(\xi ):
={(-1)^{\lfloor \alpha\rfloor}}\int_{\|\theta\|=1}((\lambda+i \xi_h\cdot\theta)^\alpha-\lambda^\alpha)d\theta .
\end{align}
To explicitly write the HFD schemes, we substitute \eqref{eq:2.2}\, \eqref{eq:2.3}, \eqref{eq:02.9} and \eqref{eq:2.8} into \eqref{eq:2.9} and set $\eta=\xi h$.  This establishes the following discrete TFL
\begin{align}\label{eq:2.13}
(-\Delta_h)_{\lambda}^{\frac{\alpha}{2}}u(x)&=\frac{1}{(2\pi)^d}\int_{D_h}S^{\alpha}_{\lambda,h}(\xi )\left(h^d\sum_{y\in h\mathbb{Z}^d}u(y)e^{-i\xi\cdot y}\right)e^{i\xi \cdot x}d\xi\nonumber\\
&=\frac{1}{h^{\alpha}}\sum_{j\in \mathbb{Z}^d}\left(\frac{h^{d+\alpha}}{(2\pi)^d}\int_{D_h}S^{\alpha}_{\lambda,h}(\xi )e^{-i\xi \cdot jh}d\xi\right) u(x+jh)\nonumber\\
& =\frac{1}{h^{\alpha}}\sum_{j\in \mathbb{Z}^d}\left(\frac{(-1)^{\lfloor \alpha\rfloor}h^{d+\alpha}}{(2\pi)^d}\int_{D_h}\int_{\|\theta\|=1}((\lambda+i \xi_h\cdot\theta)^\alpha-\lambda^\alpha)d\theta e^{-i\xi \cdot jh}d\xi \right)u(x+jh)\nonumber\\
& =\frac{1}{h^{\alpha}}\sum_{j\in \mathbb{Z}^d}\left(\frac{(-1)^{\lfloor \alpha\rfloor}}{(2\pi)^d}\int_{[-\pi,\pi]^d}\int_{\|\theta\|=1}\left(h\lambda+i {\phi_h(\eta)}\cdot\theta\right)^\alpha-(h\lambda)^\alpha)d\theta e^{-ij\cdot\eta}d\eta \right)u(x+jh),
\end{align}
where $\phi_h(\eta)=\left(\phi_h\left(\eta_1\right),\ldots,\phi_h\left(\eta_d\right)\right)$, and for $l=1,\ldots, d,$ 
\begin{align}\label{eq:02.13}
\phi_h(\eta_l):=h
\begin{cases}\psi_h\left(\frac{\eta_l}{h}\right)^{\frac{1}{2}},\ \ \eta_l\geq0,\\
-\psi_h\left(\frac{\eta_l }{h}\right)^{\frac{1}{2}},\ \ \eta_l<0.
\end{cases}
\end{align}

Furthermore, define the generating function  $$g^{(\alpha,h\lambda)}(\eta):=\frac{(-1)^{\lfloor \alpha\rfloor}}{(2\pi)^d}\int_{\|\theta\|=1}\left(h\lambda+i {\phi_h(\eta)}\cdot\theta\right)^\alpha-(h\lambda)^\alpha d\theta.$$ Then the discrete operator in \eqref{eq:2.13} can be written
\begin{align}\label{eq:2.14}
    (-\Delta_h)_{\lambda}^{\frac{\alpha}{2}}u(x)= \frac{1}{h^{\alpha}}\sum_{j\in \mathbb{Z}^d}a_{j}^{(\alpha,h\lambda)}u(x+jh),
\end{align}
where
\begin{align}\label{eq:12.15}
a_{j}^{(\alpha,h\lambda)}=\int_{[-\pi,\pi]^d}g^{(\alpha,h\lambda)}(\eta) e^{-i j\cdot\eta}d\eta,\ \ j=(j_1,\ldots, j_d),
\end{align}
and
\begin{align}\label{eq:02.16}
g^{(\alpha,h\lambda)}(\eta)&=\frac{(-1)^{\lfloor \alpha\rfloor}}{(2\pi)^d}\int_{\|\theta\|=1}\left(h\lambda+i {\phi_h(\eta)}\cdot\theta\right)^\alpha-(h\lambda)^\alpha d\theta\nonumber\\
&=\frac{(-1)^{\lfloor \alpha\rfloor}}{(2\pi)^d}\int_{\|\theta\|=1}\left(h\lambda+i {\phi_h(\eta)}\cdot\theta\right)^\alpha d\theta-(h\lambda)^\alpha\frac{2 \pi^{\frac{d}{2}}}{\Gamma\left(\frac{d}{2}\right)}\\
&=\frac{(-1)^{\lfloor \alpha\rfloor}}{(2\pi)^d}\begin{cases}
(h\lambda+i \phi_h(\eta_1))^\alpha +(h\lambda-i \phi_h(\eta_1))^\alpha-2(h\lambda)^\alpha,\ \  d = 1,\nonumber\\
\int_{0}^{2\pi}(h\lambda+i \phi_h(\eta_1)\cos\theta_1+i\phi_h(\eta_2)\sin\theta_1)^\alpha d\theta_1-2\pi(h\lambda)^\alpha,\ \ d =2,\nonumber\\
\int_{0}^{2\pi}\int_{0}^{\pi}(h\lambda+i \phi_h(\eta_1)\cos\theta_2\sin\theta_1+i\phi_h(\eta_2)\sin\theta_2 \sin \theta_1\nonumber\\ +i\phi_h(\eta_3)\cos\theta_1)^\alpha \sin \theta_1d\theta_1d \theta_2
-4\pi(h\lambda)^\alpha,\ \ d =3.
\end{cases}
\end{align}

The above analysis process can be summarized as the following algorithm for the discrete TFL.

\begin{algorithm}[H]
  \caption{The HFD for the TFL}
  \label{alg:1}
\begin{itemize}
    \item{\bf{Input: }}$\lambda$: tempering parameter, $\alpha$: fractional power, $L$: interval length, $h$: spatial step size, $U\in \mathbb{R}^{(N_1-1)\ldots(N_d-1)}$: A grid vector function.
\end{itemize}
\begin{itemize}
    \item{\bf{Output: }}$(-\Delta_h)_{\lambda}^{\frac{\alpha}{2}}U$:\ \ the discrete TFL grid function.
\end{itemize}
 \begin{algorithmic}
 \STATE {\textbf{Step 1: }}Calculate $\psi_h\left(\frac{\eta_l}{h}\right)$ and $\phi_h(\eta_l)=h\psi_h\left(\frac{\eta_l}{h}\right)^{\frac{1}{2}}$ using \eqref{eq:2.6} and \eqref{eq:02.13}.
 \STATE {\textbf{Step 2: }}Evaluate Gauss-Legendre quadrature points and weights for $N_G$ nodes using the MATLAB script legs.m, incorporate the results from Step 1, and apply the Gauss-Legendre quadrature formula to compute $g^{(\alpha,h\lambda)}(\eta)$ as defined in \eqref{eq:02.16} (for details of legs.m, refer to https://blogs.ntu.edu.sg/wanglilian/book/).
 \STATE {\textbf{Step 3: }}Use \(\texttt{fftn}\) to compute the approximating coefficients $\tilde{a}_j^{(\alpha,h\lambda)}$ of $a_j^{(\alpha,h\lambda)}$ based on the periodic function $g^{(\alpha,h\lambda)}(\eta)$, as follows $$a_j^{(\alpha,h\lambda)}\approx\tilde{a}_ j^{(\alpha,h\lambda)}=\frac{(-1)^{\lfloor{\alpha}\rfloor}}{(N_f)^d}\sum\limits_{r_1=0}^{N_f-1}\cdots\sum\limits_{r_d=0}^{N_f-1}g^{(\alpha,h\lambda)}(r_1 h_F, \ldots,r_dh_F)e^{-i\sum\limits_{l=1}^d{j_l r_lh_F}},$$
 where
 \begin{itemize}
 \item $N_f$ denotes the number of uniform grid points in the interval $[0,2\pi)$,
 \item $h_F:=\frac{2\pi}{N_f}$ is the grid step size for the FFT,
 \item $r_lh_F$ stands for the grid point corresponding to the index $r_l$, where $r_l =0,\ldots,N_f-1$,
 \item $j=(j_1,\ldots,j_d),\ 0\leq j_l\leq  N_l-1$ with $N_l=\frac{L}{h}$ and $\max\{N_l\mid l = 1, \ldots,d\}\leq{N_f}$.
 \end{itemize}
 \STATE {\textbf{Step 4: }}Compute $(-\Delta_h)_{\lambda}^{\frac{\alpha}{2}}U$ as $AU$ via the fast matrix-vector algorithm, where $A$ is a block matrix assembled from the coefficients $a_j^{(\alpha,\lambda)}$.
  \end{algorithmic}
\end{algorithm}
\begin{remark} In Step 1 of \cref{alg:1}, we only need to calculate $\phi_h(\eta_l)=h\psi_h\left(\frac{\eta_l}{h}\right)^{\frac{1}{2}}$. The computation can be simplified by noting the symmetry of the integration region in \eqref{eq:02.16} and that $\phi_h(\eta_l)$ in \eqref{eq:02.13} is an odd function, which together imply that $\int_{\|\theta\|=1}\left(h\lambda+
i{\phi_h(\eta)}\cdot\theta\right)^\alpha d\theta=\int_{\|\theta\|=1}\left(h\lambda+
i{\phi_h(-\eta)}\cdot\theta\right)^\alpha d\theta$. Henceforth, we use an equal number of grid cells in all directions, i.e., $N_1=\ldots=N_d$. The spatial step size is
$h=\frac{L}{N_1}$.
\end{remark}

\subsection{Convergence analysis for the discrete TFL}
 The convergence analysis for the HFD schemes \eqref{eq:2.9} relies on \cref{l1} below.
\begin{lemma}\label{l1}(Symbol estimates) For the frequency-domain symbol $S^{\alpha}_{\lambda}(\xi)$ of the TFL in \eqref{eq:1.5} and its discrete counterpart $S^{\alpha}_{\lambda,h}(\xi)$ in \eqref{eq:2.8}, there exists a positive constant $C$, independent of the spatial step size $h$, such that the following estimates hold
\begin{align}
    \label{eq:02.14}
    &|S^{\alpha}_{\lambda}(\xi)|\leq 2\left(\lambda^2+|\xi|^2\right)^{\frac{\alpha}{2}},\\ \label{eq:02.15}
   & |S^{\alpha}_{\lambda,h}(\xi)|\leq 2\left(\lambda^2+|\xi|^2\right)^{\frac{\alpha}{2}},\\ \label{eqq:2.16}
   & |S^{\alpha}_{\lambda}(\xi)-S^{\alpha}_{\lambda,h}(\xi)|\leq C\alpha\left(\lambda^2+|\xi|^2\right)^{\frac{\alpha+p}{2}} h^p,
    \end{align}
    where $p=4, 6, 8.$
\end{lemma}
\begin{proof}
Using the triangle inequality and $\lambda\geq 0$,  we can deduce from \eqref{eq:1.5} that
\begin{align}
|S^{\alpha}_{\lambda}(\xi)| &=\left|(-1)^{\lfloor \alpha\rfloor}\int_{\|\theta\|=1}\left(\left(\lambda+i \xi\cdot\theta\right)^\alpha-\lambda^\alpha\right)d\theta\right| \nonumber\\
&\leq \int_{\|\theta\|=1}|\lambda+i \xi\cdot\theta|^\alpha+\lambda^\alpha d\theta\nonumber\\
&\leq2\int_{\|\theta\|=1}\left(\lambda^2+ |\xi\cdot \theta|^2\right)^{\frac{\alpha}{2}}d\theta\nonumber\\
&\leq 2\left(\lambda^2+|\xi|^2\right)^{\frac{\alpha}{2}}.
\end{align}
We now prove the boundedness of the discrete frequency-domain symbol $S^{\alpha}_{\lambda,h}(\xi)$ in \eqref{eq:2.8}.

The first step is to show that  $\left|\frac{\xi_h}{\xi}\right|^2\leq 1$.
Applying \eqref{eq:02.9} and setting $p=4$ in \eqref{eq:2.6} yields the equality
$$\left|\frac{\xi_h}{\xi}\right|^2 =\frac{\psi_h(\xi)}{\xi^2}=\frac{1}{(\xi h)^2}\left(\frac{1}{6}\cos(2\xi h)-\frac{8}{3}\cos(\xi h)+\frac{5}{2}\right).$$
Introducing $y=\xi h$, we define the function $g(y)=y^2-\left(\frac{1}{6}\cos(2y)-\frac{8}{3}\cos(y)+\frac{5}{2}\right)$ and then we have
\begin{align}\label{eq:2.15}
g'(y) = 2y+\frac{1}{3}\sin(2y)-\frac{8}{3}\sin(y),\ \
g''(y) = \frac{4}{3}(\cos(y)-1)^2.
\end{align}

For $d=1$ and $\xi\in \left[-\frac{\pi}{h},\frac{\pi}{h}\right]$, there is $y\in[-\pi,\pi].$
As $g''(y)\geq0, \ \forall y\in\mathbb{R}$ in \eqref{eq:2.15}, $g'(y)$ is monotonically increasing on $\mathbb{R}$. Given that $g'(-\pi)= -2\pi, g'(0)=0$ and $ g'(\pi) = 2\pi$, it follows that $g(y)$ is monotonically decreasing on $[-\pi,0]$ and monotonically increasing on $(0,\pi]$. Thus, $g(0)$ is the minimum value on $[-\pi,\pi]$, and $g(y)\geq g(0)=0$, implying that $\left|\frac{\xi_h}{\xi}\right|^2\leq 1$.
Similarly, for $d=2, 3$, the analysis is carried out by decomposing $\xi$
into its components, and the inequality $\left|\frac{\xi_h}{\xi}\right|^2\leq 1$ still holds. Moreover, for $p=6, 8$, the same methodology shows that the inequality $\left|\frac{\xi_h}{\xi}\right|^2\leq 1$ also remains valid for the discrete symbols given in \eqref{eq:2.6}.

Leveraging $\left|\frac{\xi_h}{\xi}\right|^2\leq 1$, $\lambda\geq 0$, and the triangle inequality, we derive
\begin{align}
|S^{\alpha}_{\lambda,h}(\xi ) |
&=\left|{(-1)^{\lfloor \alpha\rfloor}}\int_{\|\theta\|=1}((\lambda+i\xi_ h\cdot\theta)^\alpha-\lambda^\alpha)d\theta\right|\nonumber\\
&\leq \int_{\|\theta\|=1}|\lambda+i\xi_h\cdot\theta|^\alpha+\lambda^\alpha d\theta\nonumber\\
& \leq \int_{\|\theta\|=1}\left(\lambda^2+ \left|\frac{\xi_ h}{\xi}\right|^2|\xi|^2\right)^{\frac{\alpha}{2}}+\lambda^\alpha d\theta
\nonumber\\
&\leq 2\left(\lambda^2+|\xi|^2\right)^{\frac{\alpha}{2}}.
\end{align}

 Before analyzing the error between  $S^{\alpha}_{\lambda}(\xi)$  and $S^{\alpha}_{\lambda,h}(\xi)$, we first establish the approximation accuracy of $\xi_h$ to $\xi$.

For $p=4, 6, 8,$ the Taylor expansion of $\cos(x)$ with the Lagrange remainder can be written as $\cos(x)=\sum\limits_{m=0}^{p/2}\frac{(-1)^mx^{2m}}{(2m)!}+\frac{(-1)^{\frac{p}{2}+1}\cos(\gamma x)x^{p+2}}{(p+2)!},$ where $\gamma\in(0, 1)$. Furthermore, setting $x=\xi h$ in \eqref{eq:2.6}, we obtain
\begin{align}\label{eq:02.7}
 \psi_h(\xi) = \xi^2 \left(1+O\left( (\xi h)^{p}\right) \right).
\end{align}
By taking the square root of both sides of \eqref{eq:02.7}, utilizing the expansion $\sqrt{1+x}=1+\frac{1}{2}x+O(x^2)$,  and substituting $x=\xi h$ into the expansion, we get
\begin{align}\label{e2.17}
\xi_h=\xi + O(\xi^{p+1}h^p).
\end{align}
Then, we derive the error order between  $S^{\alpha}_{\lambda}(\xi)$ and $S^{\alpha}_{\lambda,h}(\xi)$. Given the function $s(z)=z^{\alpha}$, where $\alpha\in (0,1)\cup(1,2)$ and $ z\in\mathbb{C}$, we have $s'(z)=\alpha z^{\alpha-1}$. The modulus $|s'(z)|$, as a function of $|z|$ for $|z|>0$, is monotonically decreasing for $\alpha\in(0,1)$ and monotonically increasing for $\alpha\in(1,2)$. By applying the fundamental identity to $s(z)$ and combining the previously established monotonicity properties of $|s'(z)|$, we deduce
\begin{align}\label{eq:2.22}
|s(z_1)-s(z_2)|=\left|\int_{z_1}^{z_2}s'(z)dz\right|\leq \int_{z_1}^{z_2}|s'(z)|dz\leq \max\{|s'(z_1)|,\ |s'(z_2)|\}|z_1-z_2|.
\end{align}
Subtracting \eqref{eq:2.8} from \eqref{eq:1.5} and applying \eqref{e2.17}, \eqref{eq:2.22}, and the inequality $\left|\frac{\xi_h}{\xi}\right|^2\leq 1$, we conclude
\begin{align}
|S^{\alpha}_{\lambda}(\xi)-S^{\alpha}_{\lambda,h}(\xi)|&=\left|\int_{\|\theta\|=1}(\lambda+i\xi\cdot\theta)^{\alpha}-(\lambda+i \xi_h \cdot\theta)^{\alpha}d\theta\right|\nonumber\\
&\leq \int_{\|\theta\|=1}\left|(\lambda+i\xi\cdot\theta)^{\alpha}-(\lambda+i \xi_h \cdot\theta)^{\alpha}\right|d\theta\nonumber\\
&\leq\max\left\{\alpha\left(\lambda^2+|\xi_h|^2\right)^{\frac{{\alpha-1}}{2}},\ \alpha\left(\lambda^2+|\xi|^2\right)^{\frac{\alpha-1}{2}}\right\}|\xi_ h-\xi|\nonumber\\
&\leq C\alpha\left(\lambda^2+|\xi|^2\right)^{\frac{\alpha-1}{2}} |\xi|^{p+1}h^p\nonumber\\
& \leq C\alpha\left(\lambda^2+|\xi|^2\right)^{\frac{\alpha+p}{2}} h^p,
\end{align}
which completes the proof.
\end{proof}

Prior to establishing the convergence for the discrete TFL, we need to introduce the Barron space, the discrete norm, and the aliasing formula for the semi-discrete Fourier transform.
\begin{definition}(Barron space and maximum norm)\cite{CDH2025,HZD2021}.\\
(i): $\mathcal{B}^{s}(\mathbb{R}^d)=\left\{u\in L^1(\mathbb{R}^d): \|u\|_{\mathcal{B}^s}<\infty\right\},\   \|u\|_{\mathcal{B}^s}=\int_{\mathbb{R}^d}(1+|\xi|)^{s}|\hat{u}(\xi)|d\xi,\ s\geq 0.$\\
(ii): The grid points in domain $\Omega\subset \mathbb{R}^d$ are denoted by $\Omega_h = \{x_j = jh \mid j \in \mathbb{Z}^d, x_j \in\Omega\}$ for a spatial step size $h > 0$. The maximum norms of a grid function $U=\{u_j\}_{j\in \mathbb{Z}^d}$ with $u_j=u(x_j)$ are defined by $$ \|U\|_{l^{\infty}}=\sup\limits_{j\in\mathbb{Z}^d}|u_j|,\ \ \|U\|_{l^{\infty}(\Omega)}=\sup\limits_{x_j\in\Omega_h}|u_j|.$$
\end{definition}

\begin{lemma}\label{d1}
(Aliasing formula) \cite[Page 32, Theorem 2]{T2000}.\\
For a function $u\in L^2(\mathbb{R}^d)$ with Fourier transform $\hat{u}$ and semi-discrete Fourier transform $\check{u}$, if its first derivative is of bounded variation, then the following identity holds
\begin{align}\label{eq:2.24}
\check{u}(\xi)=\sum_{j\in\mathbb{Z}^d}\hat{u}\left(\xi+\frac{2j\pi}{h}\right),\ \ \xi\in \left[-\frac{\pi}{h},\frac{\pi}{h}\right]^d, \, h>0.
\end{align}
\end{lemma}

With the above preparations, we now prove the convergence of the discrete TFL.

\begin{theorem}\label{T2.1}(The discrete TFL errors)
Suppose that $u\in\mathcal{B}^{\beta+\alpha}(\mathbb{R}^d)$ with $\beta>0$ and $\alpha\in(0,1)\cup(1,2)$. Then the HFD schemes \eqref{eq:2.9} for the TFL \eqref{eq:4} admit the following estimates
\begin{align}\label{eq:02.28}
\left\|(-\Delta)_{\lambda}^{\frac{\alpha}{2}}U-(-\Delta_h)_{\lambda}^{\frac{\alpha}{2}}U\right\|_{l^{\infty} (\Omega)}\leq C_1h^{\min\{\beta,p\}}\|u\|_{\mathcal{B}^{\beta+\alpha}},\ \  p=4, 6, 8,
\end{align}
where $U=\{u_j\}_{j\in \mathbb{Z}^d}$ is the grid function of the function $u$ and $C_1$ denotes a positive constant dependent on $d$, $\alpha$ and $\lambda$.
\end{theorem}

\begin{proof}
By subtracting \eqref{eq:2.9} from \eqref{eq:4}, we can obtain
{\allowdisplaybreaks
\begin{align}\label{eq:2.26}
& \ \ \ \ (-\Delta)_{\lambda}^{\frac{\alpha}{2}}u(x)-(-\Delta_h)_{\lambda}^{\frac{\alpha}{2}}u(x)\nonumber\\
& =\frac{1}{(2\pi)^d}\int_{\mathbb{R}^d}S^{\alpha}_{\lambda}(\xi)\hat{u}(\xi)e^{i \xi\cdot x}d\xi-\frac{1}{(2\pi)^d}\int_{{D_h}}S^{\alpha}_{\lambda,h}(\xi)\check{u}(\xi)e^{i\xi\cdot x}d\xi\nonumber\\
&=\frac{1}{(2\pi)^d}\int_{{D}_h }\left(S^{\alpha}_{\lambda}\left(\xi\right)-S^{\alpha}_{\lambda,h}\left(\xi\right)\right)\hat{u}(\xi)e^{i \xi\cdot x}d\xi+\frac{1}{(2\pi)^d}\int_{{D_h}}S^{\alpha}_{\lambda,h}(\xi)\left(\hat{u}(\xi)-\check{u}(\xi)\right)e^{i \xi\cdot x}d\xi\nonumber\\
&\ \ \ \ +\frac{1}{(2\pi)^d}\int_{\mathbb{R}^d\backslash{D}_h }S^{\alpha}_{\lambda}(\xi)\hat{u}(\xi)e^{i \xi\cdot x}d\xi\nonumber\\
& = J_1+J_2+J_3.
\end{align}}

We will estimate each item in turn.
$J_1$ can be estimated by using inequality \eqref{eqq:2.16} from \cref{l1} as follows
\begin{align}\label{eq:2.27}
J_1 &= \frac{1}{(2\pi)^d}\int_{{D}_h }\left(S^{\alpha}_{\lambda}(\xi)-S^{\alpha}_{\lambda,h}(\xi)\right)\hat{u}(\xi)e^{i \xi\cdot x}d\xi\nonumber\\
&\leq \frac{1}{(2\pi)^d}\int_{{D}_h }\left|S^{\alpha}_{\lambda}(\xi)-S^{\alpha}_{\lambda,h}(\xi)\right|\cdot\left|\hat{u}(\xi)\right|d\xi\nonumber\\
&\leq C\alpha h^p\int_{\mathbb{R}^d}\left(\lambda^2+|\xi|^2\right)^{\frac{\alpha+p}{2}}|\hat{u}(\xi)| d\xi \nonumber\\
& \leq C\alpha h^p\int_{\mathbb{R}^d}\left(\lambda^2+|\xi|^2\right)^{\frac{\alpha+\beta}{2}}\left(\lambda^2+|\xi|^2\right)^{\frac{p-\beta}{2}}|\hat{u}(\xi)| d\xi \nonumber\\
&  \leq C\alpha h^p\int_{\mathbb{R}^d}\left(1+|\xi|\right)^{{\alpha+\beta}}\left(1+|\xi|^2\right)^{\frac{p-\beta}{2}}|\hat{u}(\xi)| d\xi .
\end{align}
When $\beta\leq p$, applying the inequality $|\xi|\leq  \frac{d\pi}{h}$ to \eqref{eq:2.27} allows us to deduce that
\begin{align}\label{eq:2.28} J_1\leq C_1 h^\beta\|u\|_{\mathcal{B}^{\beta+\alpha}}.
\end{align}
When $\beta>p$, using the inequality $\frac{1}{1+|\xi|}\leq 1$ gives
\begin{align}\label{eq:2.29}
J_1&\leq C\alpha h^p\int_{\mathbb{R}^d}(1+|\xi|)^{{\alpha+\beta}}\left(1+|\xi|^2\right)^{\frac{p-\beta}{2}}|\hat{u}(\xi)| d\xi \leq C _1h^p \|u\|_{\mathcal{B}^{\beta+\alpha}}.
\end{align}
From \eqref{eq:2.28} and \eqref{eq:2.29}, it follows that $J_1$ satisfies the following estimate
\begin{align}\label{eq:2.30}
J_1\leq Ch^{\min\{\beta,p\}} \|u\|_{\mathcal{B}^{\beta+\alpha}}.
\end{align}

For the term $J_2$, applying \eqref{eq:02.15} from \cref{l1} together with \cref{d1} yields
{\allowdisplaybreaks
\begin{align}\label{eq:2.31}
J_2 &= \frac{1}{(2\pi)^d}\int_{{D_h}}S^{\alpha}_{\lambda,h}(\xi)(\hat{u}(\xi)-\check{u}(\xi))e^{i \xi\cdot x}d\xi\nonumber\\
&\leq \frac{1}{(2\pi)^d}\int_{{D_h}}\left|S^{\alpha}_{\lambda,h}(\xi)\right|\cdot|\hat{u}(\xi)-\check{u}(\xi)|d\xi\nonumber\\
    & \leq C \int_{{D_h}}\left(\lambda^2+|\xi|^2\right)^{\frac{\alpha}{2}}\left(\sum_{k\in \frac{2\pi}{h} \mathbb{Z}^d\backslash\{0\}}|\hat{u}(\xi+k)|\right)d\xi\nonumber\\
    &\leq C \int_{{D_h}}\left(\lambda^2+|\xi|^2\right)^{\frac{\alpha}{2}}\max_{k\in \frac{2\pi}{h} \mathbb{Z}^d\backslash\{0\}}|\xi+k|^{-\beta}\left(\sum_{k\in \frac{2\pi}{h} \mathbb{Z}^d\backslash\{0\}}|\xi+k|^\beta|\hat{u}(\xi+k)|\right)d\xi\nonumber\\
    &\leq C \left(\frac{h}{\pi}\right)^\beta\int_{{D_h}}\left(\lambda^2+|\xi|^2\right)^{\frac{\alpha}{2}}\left(\sum_{k\in \frac{2\pi}{h} \mathbb{Z}^d\backslash\{0\}}|\xi+k|^\beta|\hat{u}(\xi+k)|\right)d\xi\nonumber\\
    &\leq C \left(\frac{h}{\pi}\right)^\beta\sum_{k\in \frac{2\pi}{h} \mathbb{Z}^d\backslash\{0\}}\int_{{D_h}+k}\left(\lambda^2+|\zeta-k|^2\right)^{\frac{\alpha}{2}}|\zeta|^\beta|\hat{u}(\zeta)|d\zeta\nonumber\\
    & \leq C_1 h^\beta\int_{\mathbb{R}^d}\left(\lambda^2+|\zeta|^2\right)^{\frac{\alpha}{2}}|\zeta|^\beta|\hat{u}(\zeta)|d\zeta\nonumber\\
    & \leq C_1 h^\beta\int_{\mathbb{R}^d}\left(\lambda^2+|\zeta|^2\right)^{\frac{\alpha+\beta}{2}}|\hat{u}(\zeta)|d\zeta\nonumber\\
    &\leq C_1 h^\beta\int_{\mathbb{R}^d}\left(1+|\zeta|\right)^{\alpha+\beta}|\hat{u}(\zeta)|d\zeta.
\end{align}}

Employing  \eqref{eq:02.14} from \cref{l1} and the fact that $|\xi|\geq  \frac{\pi}{h}$ for  $\xi\in\mathbb{R}^d\backslash D_h$, we have
\begin{align}\label{eq:2.32}
J_3& = \frac{1}{(2\pi)^d}\int_{\mathbb{R}^d\backslash{D}_h }S^{\alpha}_{\lambda}(\xi)\hat{u}(\xi)e^{i \xi\cdot x}d\xi\nonumber\\
& \leq C\int_{\mathbb{R}^d\backslash{D}_h }\left|S^{\alpha}_{\lambda}(\xi)\right|\cdot|\hat{u}(\xi)|d\xi\nonumber\\
&\leq C\int_{\mathbb{R}^d\backslash{D}_h }\left(\lambda^2+|\xi|^2\right)^{\frac{\alpha}{2}}|\hat{u}(\xi)|d\xi\nonumber\\
& \leq C_1\int_{\mathbb{R}^d\backslash{D}_h }(1+|\xi|)^{\beta+\alpha}|\xi|^{-\beta}|\hat{u}(\xi)|d\xi\nonumber\\
& \leq C_1 h^\beta \|u\|_{\mathcal{B}^{\beta+\alpha}}.
\end{align}

From \eqref{eq:2.30}\--\eqref{eq:2.32}, we thus conclude the proof of \cref{T2.1}.
\end{proof}

\section{The HFD methods for TFL equations}
 We construct the HFD methods for \eqref{eq:1.1} and \eqref{eq:1.2} through the operator approximations given in \eqref{eq:2.1} and \eqref{eq:2.9}.
\begin{align}\label{eq:3.1}
(-\Delta_h)_{\lambda}^{\frac{\alpha}{2}}u_j^h-\sigma \Delta_hu_j^h+\nu u_j^h &= f(x_j),\ \ x_j\in\Omega_h ,\ \ \alpha\in(0,1)\cup(1,2), \\ \label{eq:3.2}
 u_j^h &= 0,\ \ x_j\in \Omega_h^c:= h\mathbb{Z}^d\backslash\Omega_h,
\end{align}
where $u_j^h=u^h(x_j)$ represents the numerical solution.

To analyze the stability and convergence of \eqref{eq:3.1} and \eqref{eq:3.2}, we first recall several key definitions and lemmas.

\begin{definition}(Discrete inner product and norm of basic function spaces)\cite{CDH2025,HZD2021}.\label{0d1}\\
For grid functions $U=\{u_j\}_{j\in\mathbb{Z}^d}$ and $ V=\{v_j\}_{j\in\mathbb{Z}^d}$ with $u_j=u(x_j), v_j=v(x_j)$, the definitions of discrete inner product and norm are as follows.\\
(i): The discrete inner product is defined by $(U,V)_h=h^d\sum\limits_{j\in \mathbb{Z}^d}u_j\overline{v_j},$ where $\overline{v_j}$ denotes the complex conjugate of $v_j$. For a bounded domain $\Omega\subset\mathbb{R}^d$, $(U,V)_{\Omega}=h^d\sum_{x_j\in\Omega_h}u_j\overline{v_j}.$ \\
(ii): The discrete $l^2$-norms read $\|U\|_{l^2}=\sqrt{(U,U)_h}$ and $\|U\|_{l^2(\Omega)}=\sqrt{(U,U)_\Omega}.$\\
(iii): Parseval's identity is given by $(U,V)_h=\frac{1}{(2\pi)^d}\int_{D_h} \check{{u}}(\xi)\overline{\check{{v}}(\xi)}d\xi,$ where $\check{{u}}(\xi)$ and $\check{{v}}(\xi)$ are the semi-discrete Fourier transforms of $u$ and $v$, respectively.\\
(iv): For $s>0$, the discrete fractional Sobolev semi-norm $|\cdot|_{H_h^s}$ and the corresponding norm $\|\cdot\|_{H_h^s}$ are defined by
$$|U|^2_{H_h^s}=\int_{D_h}|\xi|^{2s}|\check{u}(\xi)|^2d\xi, \ \ \|U\|^2_{H_h^s}=\int_{D_h}\left(1+|\xi|^{2s}\right)|\check{u}(\xi)|^2d\xi.$$
(v): The basic function spaces are specified as
$$H_h^s:=\left\{U\ |\ \|U\|_{H_h^s}<\infty\right\},\ \ V_h^s:=\left\{U\ |\ U\in H_h^s,\  u(x_j)=0,\ x_j\in \Omega_h^c\right\}.$$
\end{definition}

\begin{lemma}(Poincar$\acute{\text{e}}$'s inequality)\cite{CDH2025,HZD2021}.\label{l3.1}
For grid functions $U\in H_h^s$ with $s>0$, it holds that
\begin{align}
\|U\|_{l^2} \leq C |U|_{H_h^s},
\end{align}
where $C$ is a positive constant independent of $h$.
\end{lemma}

Next, we show the positivity of the discrete TFL under the discrete inner product.
\begin{lemma}(Positivity).\label{l3.2}
Assume that $U\in H_h^s$ for $s>0$. Then the following estimate holds
\begin{align}
\left(\left(-\Delta_h\right)_{\lambda}^{\frac{\alpha}{2}}U,U\right)_h>0,\ \ \alpha\in(0,1)\cup(1,2),
\end{align}
for any $h>0$.
\end{lemma}
\begin{proof}
 When $d=2$, with $\xi=(\xi_1,\xi_2)$ and the unit vector $ \theta=(\cos\theta_1,\sin\theta_1)$, we take the discrete inner product of the grid functions $(-\Delta_h)_{\lambda}^{\frac{\alpha}{2}}U$ and $U$. Together with \eqref{eq:02.9} and \eqref{eq:2.9}, this leads to the following equality
\begin{align}\label{e:3.5}
&\left((-\Delta_h)_{\lambda}^{\frac{\alpha}{2}}U,U\right)_h
\nonumber\\
&=\frac{1}{(2\pi)^d}\int_{D_h}\mathcal{F}_h\left[(-\Delta_h)_{\lambda}^{\frac{\alpha}{2}}U\right](\xi)\cdot \overline{\mathcal{F}_h\left[U\right](\xi)}d\xi\nonumber\\
&=\frac{(-1)^{\lfloor\alpha\rfloor}}{(2\pi)^{d}}\int_{\|\theta\|=1}\int_{D_h\cap\{\xi_1>0,\xi_2=0\}}\left(\left(\lambda+i\left(\psi_h\left(\xi_1\right)^{\frac{1}{2}},\psi_h\left(\xi_2\right)^{\frac{1}{2}}\right)\cdot\theta\right)^{\alpha}-\lambda^\alpha\right) \left|\check{u}(\xi)\right|^2d\xi d\theta\nonumber\\
&\ \ \ \  +\frac{(-1)^{\lfloor\alpha\rfloor}}{(2\pi)^{d}}\int_{\|\theta\|=1}\int_{D_h\cap\{\xi_1>0,\xi_2>0\}}\left(\left(\lambda+i\left(\psi_h\left(\xi_1\right)^{\frac{1}{2}},\psi_h\left(\xi_2\right)^{\frac{1}{2}}\right)\cdot\theta\right)^{\alpha}-\lambda^\alpha\right) \left|\check{u}(\xi)\right|^2d\xi d\theta\nonumber\\
&\ \ \ \  + \frac{(-1)^{\lfloor\alpha\rfloor}}{(2\pi)^{d}}\int_{\|\theta\|=1}\int_{D_h\cap\{\xi_1>0,\xi_2<0\}}\left(\left(\lambda+i\left(\psi_h\left(\xi_1\right)^{\frac{1}{2}},-\psi_h\left(\xi_2\right)^{\frac{1}{2}}\right)\cdot\theta\right)^\alpha-\lambda^\alpha\right) \left|\check{u}(\xi)\right|^2d\xi d\theta\nonumber\\
&\ \ \ \  + \frac{(-1)^{\lfloor\alpha\rfloor}}{(2\pi)^{d}}\int_{\|\theta\|=1}\int_{D_h\cap\{\xi_1<0,\xi_2=0\}}\left(\left(\lambda+i\left(-\psi_h\left(\xi_1\right)^{\frac{1}{2}},\psi_h\left(\xi_2\right)^{\frac{1}{2}}\right)\cdot\theta\right)^\alpha-\lambda^\alpha\right) \left|\check{u}(\xi)\right|^2d\xi d\theta\nonumber\\
&\ \ \ \  + \frac{(-1)^{\lfloor\alpha\rfloor}}{(2\pi)^{d}}\int_{\|\theta\|=1}\int_{D_h\cap\{\xi_1<0,\xi_2<0\}}\left(\left(\lambda+i\left(-\psi_h\left(\xi_1\right)^{\frac{1}{2}},-\psi_h\left(\xi_2\right)^{\frac{1}{2}}\right)\cdot\theta\right)^\alpha-\lambda^\alpha\right) \left|\check{u}(\xi)\right|^2d\xi d\theta\nonumber\\
&\ \ \ \  + \frac{(-1)^{\lfloor\alpha\rfloor}}{(2\pi)^{d}}\int_{\|\theta\|=1}\int_{D_h\cap\{\xi_1<0,\xi_2>0\}}\left(\left(\lambda+i\left(-\psi_h\left(\xi_1\right)^{\frac{1}{2}},\psi_h\left(\xi_2\right)^{\frac{1}{2}}\right)\cdot\theta\right)^\alpha-\lambda^\alpha\right) \left|\check{u}(\xi)\right|^2d\xi d\theta.
\end{align}
For $d=3$, with $\xi=(\xi_1,\xi_2, \xi_3)$ and the unit vector $\theta=(\cos\theta_2\sin\theta_1, \sin\theta_2\sin\theta_1, \cos\theta_1)$, the equality \eqref{e:3.5} can be decomposed into fourteen cases.

As $u$ is real-valued, it follows from \eqref{eq:2.2} that
\begin{align}\label{eq:3.6}
\check{u}(-\xi)=\overline{\check{u}(\xi)}, \ \ \left|\check{u}(-\xi)\right|=\left|\overline{\check{u}(\xi)}\right|=|{\check{u}(\xi)}|.
\end{align}
Substituting  $-\xi$ for $\xi$ in the last three lines of \eqref{e:3.5} and applying \eqref{eq:2.6} and \eqref{eq:3.6}, we get
\begin{align}\label{eq:3.7}
\left(\left(-\Delta_h\right)_{\lambda}^{\frac{\alpha}{2}}U,U\right)_h=J_4+J_5,
\end{align}
where
\begin{align}\label{eq:3.8}
J_4=&\frac{(-1)^{\lfloor\alpha\rfloor}}{(2\pi)^{d}}\int_{\|\theta\|=1}\int_{D_h\cap\{\xi_1>0,\xi_2\geq0\}}\left(\left(\lambda+i\left(\psi_h\left(\xi_1\right)^{\frac{1}{2}},\psi_h\left(\xi_2\right)^{\frac{1}{2}}\right)\cdot\theta\right)^{\alpha}-2\lambda^\alpha\right.\nonumber\\
 &\left.+ \left(\lambda+i\left(-\psi_h\left(\xi_1\right)^{\frac{1}{2}},-\psi_h\left(\xi_2\right)^{\frac{1}{2}}\right)\cdot\theta\right)^\alpha\right) \left|\check{u}(\xi)\right|^2d\xi d\theta\\ \label{eq:3.9}
J_5=&\frac{(-1)^{\lfloor\alpha\rfloor}}{(2\pi)^{d}}\int_{\|\theta\|=1}\int_{D_h\cap\{\xi_1>0,\xi_2<0\}}\left(\left(\lambda+i\left(\psi_h\left(\xi_1\right)^{\frac{1}{2}},-\psi_h\left(\xi_2\right)^{\frac{1}{2}}\right)\cdot\theta\right)^\alpha-2\lambda^\alpha\right.\nonumber\\
&\left.+\left(\lambda+i\left(-\psi_h\left(\xi_1\right)^{\frac{1}{2}},\psi_h\left(\xi_2\right)^{\frac{1}{2}}\right)\cdot\theta\right)^\alpha\right) \left|\check{u}(\xi)\right|^2d\xi d\theta.
\end{align}
Further, leveraging the following equalities
\begin{equation}\label{eq:3.10}
z^\alpha = \begin{cases}
    \frac{\alpha}{\Gamma(1-\alpha)}\int_{0}^{\infty}(1-e^{-zy})y^{-\alpha-1}dy,\ \ &0<\alpha<1,\\
    \frac{\alpha(\alpha-1)}{\Gamma(2-\alpha)}\int_{0}^{\infty}(e^{-zy}-1+zy)y^{-\alpha-1}dy,\ \ &1<\alpha<2,
\end{cases}
\end{equation}
we can derive
\begin{align}\label{eq:3.11}
J_4=&\frac{(-1)^{\lfloor\alpha\rfloor}}{(2\pi)^{d}}\int_{\|\theta\|=1}\int_{D_h\cap\{\xi_1>0,\xi_2\geq0\}}M_1 |\check{u}(\xi)|^2d\xi d\theta,\\ \label{eq:3.12}
J_5=&\frac{(-1)^{\lfloor\alpha\rfloor}}{(2\pi)^{d}}\int_{\|\theta\|=1}\int_{D_h\cap\{\xi_1>0,\xi_2<0\}}M_2|\check{u}(\xi)|^2d\xi d\theta,
\end{align}
in which, for $0<\alpha<1$,
\begin{align}\label{eq:3.13}
M_1 &= \frac{\alpha}{\Gamma(1-\alpha)}\int_{0}^\infty\left(2-e^{-i\left(\psi_h\left(\xi_1\right)^{\frac{1}{2}},\psi_h\left(\xi_2\right)^{\frac{1}{2}}\right)\cdot\theta y}-e^{-i\left(-\psi_h\left(\xi_1\right)^{\frac{1}{2}},-\psi_h\left(\xi_2\right)^{\frac{1}{2}}\right)\cdot\theta y}\right) e^{-\lambda y}y^{-\alpha-1}dy\nonumber\\
&=\frac{\alpha}{\Gamma(1-\alpha)}\int_{0}^\infty\left(2-2\cos\left(\left(\psi_h\left(\xi_1\right)^{\frac{1}{2}},\psi_h\left(\xi_2\right)^{\frac{1}{2}}\right)\cdot\theta y\right)\right) e^{-\lambda y}y^{-\alpha-1}dy\nonumber\\
&=\frac{\alpha}{\Gamma(1-\alpha)}\int_{0}^\infty4\sin^2\left(\frac{\left(\psi_h\left(\xi_1\right)^{\frac{1}{2}},\psi_h\left(\xi_2\right)^{\frac{1}{2}}\right)\cdot\theta y}{2}\right) e^{-\lambda y}y^{-\alpha-1}dy,\\ \label{eq:3.14}
M_2 &= \frac{\alpha}{\Gamma(1-\alpha)}\int_{0}^\infty\left(2-e^{-i\left(\psi_h\left(\xi_1\right)^{\frac{1}{2}},-\psi_h\left(\xi_2\right)^{\frac{1}{2}}\right)\cdot\theta y}-e^{-i\left(-\psi_h\left(\xi_1\right)^{\frac{1}{2}},\psi_h\left(\xi_2\right)^{\frac{1}{2}}\right)\cdot\theta y}\right) e^{-\lambda y}y^{-\alpha-1}dy\nonumber\\
&=\frac{\alpha}{\Gamma(1-\alpha)}\int_{0}^\infty\left(2-2\cos\left(\left(\psi_h\left(\xi_1\right)^{\frac{1}{2}},-\psi_h\left(\xi_2\right)^{\frac{1}{2}}\right)\cdot\theta y\right)\right) e^{-\lambda y}y^{-\alpha-1}dy\nonumber\\
&=\frac{\alpha}{\Gamma(1-\alpha)}\int_{0}^\infty4\sin^2\left(\frac{\left(\psi_h\left(\xi_1\right)^{\frac{1}{2}},-\psi_h\left(\xi_2\right)^{\frac{1}{2}}\right)\cdot\theta y}{2}\right) e^{-\lambda y}y^{-\alpha-1}dy,
\end{align}
for $1<\alpha<2$,
{\allowdisplaybreaks
\begin{align}\label{eq:3.15}
M_1&=\frac{\alpha(\alpha-1)}{\Gamma(2-\alpha)}\int_{0}^\infty\left(-2+e^{-i\left(\psi_h\left(\xi_1\right)^{\frac{1}{2}},\psi_h\left(\xi_2\right)^{\frac{1}{2}}\right)\cdot\theta y}+e^{-i\left(-\psi_h\left(\xi_1\right)^{\frac{1}{2}},-\psi_h\left(\xi_2\right)^{\frac{1}{2}}\right)\cdot\theta y}\right) e^{-\lambda y}y^{-\alpha-1}dy\nonumber\\
&=\frac{\alpha(\alpha-1)}{\Gamma(2-\alpha)}\int_{0}^\infty\left(-2+2\cos\left(\left(\psi_h\left(\xi_1\right)^{\frac{1}{2}},\psi_h\left(\xi_2\right)^{\frac{1}{2}}\right)\cdot\theta y\right)\right) e^{-\lambda y}y^{-\alpha-1}dy\nonumber\\
&=-\frac{\alpha(\alpha-1)}{\Gamma(2-\alpha)}\int_{0}^\infty4\sin^2\left(\frac{\left(\psi_h\left(\xi_1\right)^{\frac{1}{2}},\psi_h\left(\xi_2\right)^{\frac{1}{2}}\right)\cdot\theta y}{2}\right) e^{-\lambda y}y^{-\alpha-1}dy,\\ \label{eq:3.16}
M_2 &= \frac{\alpha(\alpha-1)}{\Gamma(2-\alpha)}\int_{0}^\infty\left(-2+e^{-i\left(\psi_h\left(\xi_1\right)^{\frac{1}{2}},-\psi_h\left(\xi_2\right)^{\frac{1}{2}}\right)\cdot\theta y}+e^{-i\left(-\psi_h\left(\xi_1\right)^{\frac{1}{2}},\psi_h\left(\xi_2\right)^{\frac{1}{2}}\right)\cdot\theta y}\right) e^{-\lambda y}y^{-\alpha-1}dy\nonumber\\
&=\frac{\alpha(\alpha-1)}{\Gamma(2-\alpha)}\int_{0}^\infty\left(-2+2\cos\left(\left(\psi_h\left(\xi_1\right)^{\frac{1}{2}},-\psi_h\left(\xi_2\right)^{\frac{1}{2}}\right)\cdot\theta y\right)\right) e^{-\lambda y}y^{-\alpha-1}dy\nonumber\\
&=-\frac{\alpha(\alpha-1)}{\Gamma(2-\alpha)}\int_{0}^\infty4\sin^2\left(\frac{\left(\psi_h\left(\xi_1\right)^{\frac{1}{2}},-\psi_h\left(\xi_2\right)^{\frac{1}{2}}\right)\cdot\theta y}{2}\right) e^{-\lambda y}y^{-\alpha-1}dy.
\end{align}}
Combining \eqref{eq:3.11} and \eqref{eq:3.12} with \eqref{eq:3.13}--\eqref{eq:3.16} implies $\left((-\Delta_h)_{\lambda}^{\frac{\alpha}{2}}U,U\right)_h>0$, which concludes the proof.
\end{proof}
Then, we state the norm equivalence lemma of the discrete Laplacian under the discrete inner product.
\begin{lemma}\label{l3.3}(Norm equivalence). Let the grid function $U\in H_h^s$. Then, an equivalent statement is as follows
\begin{align}\label{eq:3.19}
\frac{1}{\pi^2}\left|U\right|^2_{H_h^1}\leq(-\Delta_h U,U)_h\leq \left|U\right|^2_{H_h^1},
\end{align}
for any $h>0$.
\end{lemma}
\begin{proof}
Taking the discrete inner product between $-\Delta_hU$ and $U$, and employing \eqref{eq:2.4} and  Parseval's equality (iii) from \cref{0d1}, we have
\begin{align}\left(-\Delta_h U,U\right)_h&=\frac{1}{(2\pi)^d}\int_{D_h} \mathcal{F}_h\left[-\Delta_hU\right](\xi)\cdot \overline{\mathcal{F}_h\left[U\right](\xi)}d\xi=\frac{1}{(2\pi)^d}\int_{D_h} \psi_h(\xi) |\check{u}(\xi)|^2d\xi.
\end{align}
In \cref{l1}, we have already proven that $\frac{\psi_h(\xi)}{\xi^2}\leq 1$. Thus, it remains to show that $c\leq\frac{\psi_h(\xi)}{\xi^2}$ for some positive constant $c$.

Setting $p=4$ in \eqref{eq:2.6} and letting $y=\xi h$ with $y\in[-\pi,\pi]^d$, we define the function $g_0(y)$ as
\begin{align}g_0(y)=cy^2-\psi_h\left(\frac{y}{h}\right)h^2=cy^2-\left(\frac{1}{6}\cos(2y)-\frac{8}{3}\cos(y)+\frac{5}{2}\right).
\end{align}

For $d=1$, owing to the evenness of $g_0(y)$ on $[-\pi,\pi]$,  it suffices to verify $g_0(y)\leq 0$ on $[0,\pi].$ Upon differentiating $g_0(y)$, we arrive at
\begin{align}\label{e3.64}
&g_0'(y) = 2cy+\frac{1}{3}\sin(2y)-\frac{8}{3}\sin(y),\ \ \ \
g_0''(y) = 2c+\frac{2}{3}\cos(2y)-\frac{8}{3}\cos(y),\\ \label{e3.65}
& g_0'''(y) = -\frac{4}{3}\sin(2y)+\frac{8}{3}\sin(y) = \frac{8}{3}\sin(y)(1-\cos(y)), \ \ y\in[0,\pi].
\end{align}
Setting $g_0(\pi)=0$ implies $c=\frac{16}{3\pi^2}$. It follows from \eqref{e3.65} that $g_0'''(y)\geq0$, which shows that $g_0''(y)$ is monotonically increasing on $[0,\pi]$. In conjunction with $g_0''(0)<0$ and $g_0''(\pi)>0$, we conclude that $g_0'(y)$ first decreases and then increases on $[0,\pi]$. Using $g_0'(0)=0$ and $g_0'(\pi)>0$, we conclude that $g_0(y)$ is also first decreasing and then increasing on $[0,\pi]$, whence $g_0(y)\leq \max\{g_0(0),g_0(\pi)\}=0, \ \forall y\in[-\pi,\pi]$.

Based on the preceding analysis, we confirm that
\begin{align}
\frac{1}{\pi^2}|\xi|^2\leq\frac{16}{3\pi^2}|\xi|^2\leq \psi_h(\xi)\leq |\xi|^2.
\end{align}
\eqref{eq:3.19} follows directly from (iv) of \cref{0d1}.

For $d=2, 3$, \eqref{eq:3.19} remains valid by a series of one-dimensional cases.
For $p=6, 8$, the result also holds by the same approach, albeit with a different constant $c$.
\end{proof}
\begin{remark}
 Although \cref{l1} provides an upper bound for $S^{\alpha}_{\lambda,h}$, a corresponding lower bound is non-trivial. The main difficulty stems from the fact that $S^{\alpha}_{\lambda,h}(\xi)$ is defined by an intricate integral.
\end{remark}
Finally, we prove the desired stability and convergence of the HFD methods for the TFL equations.
\begin{theorem}\label{T3.1}(Stability).
For the HFD schemes \eqref{eq:3.1} and \eqref{eq:3.2}, the numerical solution $U^h=\left\{u_j^h\right\}_{j\in \mathbb{Z}^d}$ is unique and satisfies the stability estimate
\begin{align}\label{eq:3.24}
\sigma\left|U^h\right|_{H_h^1}+\nu \left\|U^h\right\|_{l^2}\leq C \|f\|_{l^2(\Omega)},
\end{align}
for all $h>0$, where $ f=\left\{f_j\right\}_{x_j\in \Omega_h}$ and $C$ is a generic positive constant independent of $h$.
\end{theorem}
\begin{proof}
Taking the discrete inner product of \eqref{eq:3.1} with $U^h$, we obtain
\begin{align}\label{eq:3.25}
\left(\left(-\Delta_h\right)_{\lambda}^{\frac{\alpha}{2}}U^h,U^h\right)_h+\sigma\left(-\Delta_hU^h,U^h\right)_h+\nu \left\|U^h\right\|_{l^2}=\left(f,U^h\right)_\Omega.
\end{align}
 Employing \cref{l3.2}, \cref{l3.3} and the Cauchy-Schwarz inequality, we get the following inequality from \eqref{eq:3.25} \begin{align}\label{eq:3.26}
\frac{\sigma}{\pi^2}\left|U^h\right|^2_{H_h^1}+\nu \left\|U^h\right\|^2_{l^2}\leq \|f\|_{l^2(\Omega)}\left\|U^h\right\|_{l^2}.
\end{align}
Applying \cref{l3.1} to \eqref{eq:3.26} yields
\begin{align}\label{eq:3.27}
\nu\left\|U^h\right\|_{l^2}\leq \|f\|_{l^2(\Omega)},\ \ \frac{\sigma}{\pi^2}\left|U^h\right|^2_{H_h^1}\leq \|f\|_{l^2(\Omega)}\left\|U^h\right\|_{l^2}\leq C\|f\|_{l^2(\Omega)}\left\|U^h\right\|_{H_h^1}.
\end{align}
\eqref{eq:3.24} follows directly from \eqref{eq:3.27}.

If $f(x)= 0$ in \eqref{eq:3.24}, then $\left\|U^h\right\|_{l^2}=0$. Together with \eqref{eq:2.3} and (iii) of \cref{0d1}, this implies the uniqueness of the numerical solution $U^h$.
\end{proof}
\begin{theorem}\label{d3.1}(Convergence). Let ${u}\in\mathcal{B}^{\beta+\alpha}(\mathbb{R}^d)$ be the exact solution of \eqref{eq:1.1} and \eqref{eq:1.2}, with $\beta>0$ and $\alpha\in(0,1)\cup(1,2)$. For the exact solution grid function $U=\{u_j\}_{j\in\mathbb{Z}^d}$ and numerical solution $U^h=\{u^h_j\}_{j\in\mathbb{Z}^d}$ of the HFD schemes \eqref{eq:3.1} and \eqref{eq:3.2}, the following error estimates hold
 \begin{align}
 \sigma \left|U-U^h\right|_{H_h^1}+\nu \|U-U^h\|_{l^2} \leq C h^{\min\{\beta,p\}},
 \end{align}
 for $p = 4, 6, 8$ and sufficiently small $h$.
\end{theorem}
\begin{proof}
Denote $E_j=u_j-u_j^h$. Then, subtracting \eqref{eq:3.1} and \eqref{eq:3.2} from \eqref{eq:1.1} and \eqref{eq:1.2} term by term gives the error equations
\begin{align}
\left(-\Delta_h\right)_{\lambda}^{\frac{\alpha}{2}}E_j-\sigma\Delta_h E_j+\nu E_j&=\left(-\Delta_h\right)_{\lambda}^{\frac{\alpha}{2}}u_j-\left(-\Delta\right)_{\lambda}^{\frac{\alpha}{2}}u_j+\sigma\left(\Delta u_j-\Delta_hu_j\right),\ \ x_j\in \Omega_h,\nonumber\\
E_j&=0,\ \  x_j\in \Omega^c_h.
\end{align}
Defining the error grid function $E=U-U^h$ and employing an argument similar to the proof of \cref{T3.1}, we derive the following estimate
\begin{align}\label{eq:3.30}
\sigma\left|E\right|_{H_h^1}+\nu \left\|E\right\|_{l^2}\leq C \left\|\left(-\Delta_h\right)_{\lambda}^{\frac{\alpha}{2}}U-\left(-\Delta\right)_{\lambda}^{\frac{\alpha}{2}}U+\sigma\left(\Delta U-\Delta_hU\right)\right\|_{l^2(\Omega)}.
\end{align}
Applying \cref{T2.1} and \eqref{eq:2.1} to \eqref{eq:3.30} yields the key estimate
\begin{align}
\sigma\left|E\right|_{H_h^1}+\nu \left\|E\right\|_{l^2}& \leq C \left\|\left(-\Delta_h\right)_{\lambda}^{\frac{\alpha}{2}}U-\left(-\Delta\right)_{\lambda}^{\frac{\alpha}{2}}U\right\|_{l^2(\Omega)}+\sigma\left\|\Delta U-\Delta_hU\right\|_{l^2(\Omega)}\leq C h^{\min\{\beta,p\}}\|u\|_{\mathcal{B}^{\beta+\alpha}}.
\end{align}
This establishes the desired results.
\end{proof}
\begin{remark}
 The HFD methods \eqref{eq:3.1} and \eqref{eq:3.2} constitute a linear system
\begin{align}\label{eq:3.32}
AU^h=\left(A_0+{\sigma}A_1+\nu \text{I}\right)U^h=f,
\end{align}
where $A_0$ and $A_1$ are symmetric block Toeplitz matrices corresponding to the coefficients in \eqref{eq:2.14} and \eqref{eq:2.1}, respectively, $I$ is the identity matrix, $U^h$ denotes the numerical solution, and $f$ stands for the source grid function defined on $\Omega_h$. The structure of the matrix $A$ is analogous to that described in Section 4 of Ref. \cite{HZD2021}, but differs only in the coefficients. Consequently, in numerical experiments, matrix-vector multiplications employ the Fast Fourier Transform (FFT). For specific implementation details, see Ref. \cite{HZD2021}.
\end{remark}
The following algorithm implements the HFD methods for the TFL equations.
\begin{algorithm}[H]
  \caption{The HFD for the TFL equations}
  \label{alg:2}
\begin{itemize}
    \item{\bf{Input: }}$\lambda$: tempering parameter, $\alpha$: fractional power, $L$: interval length, $h$: spatial step size, $f$: the source grid function.
\end{itemize}
\begin{itemize}
    \item{\bf{Output: }}$U^h$: the numerical solution for TFL equations.
\end{itemize}
 \begin{algorithmic}
 \STATE {\textbf{Step 1: }}Generate grid $\Omega_h=\left\{x_j=x_{\min}+jh \mid j=(j_1,\ldots,j_d),\  1\leq j_l\leq N_1-1= \frac{L}{h}-1,\ l=1,\ldots,d\right\}$, with $x_{\min}$ denoting the starting coordinate in each dimension.
 \STATE {\textbf{Step 2: }}Calculate the source grid function $f=\{f(x_j)\mid x_j\in\Omega_h\}$ on $\Omega_h$.
 \STATE {\textbf{Step 3: }}Compute the first row of matrix $A$ by combining \cref{alg:1} with \eqref{eq:2.1} and \eqref{eq:3.32}.
 \STATE {\textbf{Step 4: }}Solve $AU^h=f$ using the Preconditioned Conjugate Gradient (PCG) method, as implemented in Ref. \cite{HZD2021}.
  \end{algorithmic}
\end{algorithm}
\section{Numerical accuracy}
In this section, we assess the errors and convergence rates of the proposed methods from three aspects:
(1) the approximation of the coefficients $a_{j}^{(\alpha,\lambda)}$ in \eqref{eq:12.15}; (2) the discretization of the TFL $(-\Delta)_{\lambda}^{\frac{\alpha}{2}}$ \eqref{eq:4}; and (3) the numerical solution of the TFL equations.

 For the HFD schemes \eqref{eq:2.14} with $p=4, 6, 8$, we compute the coefficients in \cref{alg:1} using $N_G=20$ Gauss-Legendre quadrature points for all numerical experiments. These experiments are conducted using MATLAB R2020b (64-bit) on a personal computer equipped with a 2.3 GHz processor and 16 GB of RAM.
 \subsection{The approximation accuracy for the coefficients $a_j^{(\alpha,h\lambda)}$}
This subsection examines the errors and convergence rates of the coefficients $a_j^{(\alpha,h\lambda)}$ in \eqref{eq:12.15}. In practice, we compute the discrete Fourier transforms $\tilde{a}_j^{(\alpha,h\lambda)}$ \cref{alg:1} via \texttt{fftn} to approximate $a_j^{(\alpha,h\lambda)}$.
\cref{tab1,tab2,tab3} summarize the mean squared error $e(N_f)$ and the spectral precision of the approximate coefficients $\tilde{a}_j^{(\alpha,h\lambda)}$ for $p=4, 6, 8$. The mean squared error $e(N_f)$ is defined as
\[
e(N_f) = \left(\frac{1}{(N_f)^d}\sum_{l=1}^{d}\sum_{j_l=0}^{N_f-1}\left|\tilde{a}_{j}^{(\alpha,h\lambda)}(N_f)-\tilde{a}_j^{(\alpha,h\lambda)}(2N_f)\right|^2\right)^{1/2},
\]
where $j=(j_1,\ldots, j_d)$; the Fourier resolution
$N_f$ takes the successive values $\{2^6, 2^7, 2^8, 2^9, 2^{10}\}$; the fractional power $\alpha$ is selected from $ \{0.4, 0.8, 1.2, 1.8\}$; and the remaining parameters are fixed: dimension $d=2$, spatial step size $h=1/32$ and tempering parameter $\lambda=0.5$.
\begin{table}[H]
\centering
  \renewcommand{\arraystretch}{1.15} 
    \setlength{\tabcolsep}{8pt}
\caption{Numerical error $e(N_f)$ and convergence rate of $\tilde{a}_j^{(\alpha,h\lambda)}$ via \texttt{fft2} for $p=4$.}
\label{tab1}
\begin{tabular}{cccccccccc}
\toprule
\multirow{2}{*}{$N_f$} &
\multicolumn{2}{c}{$\alpha=0.4$} &
\multicolumn{2}{c}{$\alpha=0.8$} &
\multicolumn{2}{c}{$\alpha=1.2$}&
\multicolumn{2}{c}{$\alpha=1.8$}\\
\cmidrule(lr){2-3} \cmidrule(lr){4-5} \cmidrule(lr){6-7} \cmidrule(lr){8-9}
 & $e(N_f)$ & rate & $e(N_f)$ & rate & $e(N_f)$ & rate& $e(N_f)$ &rate \\
\midrule
$2^6$ & 6.29e-05 & * & 8.69e-06  & * & 2.24e-06  & * & 3.58e-07  & * \\
$2^7$ & 8.30e-06   &  2.92  &   9.11e-07  & 3.25 & 1.84e-07 &  3.60& 2.02e-08 & 4.15 \\
$2^8$ & 5.88e-07    &   3.82&  5.07e-08 & 4.17  & 8.02e-09 &  4.52 &  6.02e-10 &  5.07\\
$2^9$ & 1.45e-08 &  5.34   &   9.73e-10 &   5.70& 1.19e-10 &  6.07 & 6.09e-12 & 6.63\\
$2^{10}$ &4.75e-11  &  8.25  & 2.47e-12 &  8.62  &  2.34e-13 &  9.00 & 8.07e-15 & 9.56 \\
\bottomrule
\end{tabular}
\end{table}
\begin{table}[H]
\centering
  \renewcommand{\arraystretch}{1.15} 
    \setlength{\tabcolsep}{8pt}
\caption{Numerical error $e(N_f)$ and convergence rate of $\tilde{a}_j^{(\alpha,h\lambda)}$ via \texttt{fft2} for $p=6$}
\label{tab2}
\begin{tabular}{cccccccccc}
\toprule
\multirow{2}{*}{$N_f$} &
\multicolumn{2}{c}{$\alpha=0.4$} &
\multicolumn{2}{c}{$\alpha=0.8$} &
\multicolumn{2}{c}{$\alpha=1.2$}&
\multicolumn{2}{c}{$\alpha=1.8$}\\
\cmidrule(lr){2-3} \cmidrule(lr){4-5} \cmidrule(lr){6-7} \cmidrule(lr){8-9}
 & $e(N_f)$ & rate & $e(N_f)$ & rate & $e(N_f)$ & rate& $e(N_f)$ & rate \\
\midrule
$2^6$ & 6.44e-05 & * & 8.89e-06 & * & 2.29e-06  & * & 3.67e-07   & * \\
$2^7$ & 8.41e-06  &   2.94  &    9.26e-07  & 3.26 &1.88e-07  &  3.61  &   2.07e-08     &  4.15  \\
$2^8$ & 5.98e-07    &    3.81   &  5.16e-08 & 4.16  &  8.17e-09 &   4.52   &   6.13e-10   &  5.08\\
$2^9$ & 1.50e-08  &   5.32     &  1.01e-09  &    5.68 &  1.24e-10& 6.04  &   6.34e-12  &  6.60 \\
$2^{10}$ & 4.76e-11   &  8.30  & 2.47e-12  &  8.68    &  2.34e-13 &  9.05& 8.07e-15 & 9.62  \\
\bottomrule
\end{tabular}
\end{table}
\begin{table}[H]
\centering
  \renewcommand{\arraystretch}{1.15} 
    \setlength{\tabcolsep}{8pt}
\caption{Numerical error $e(N_f)$ and convergence rate of $\tilde{a}_j^{(\alpha,h\lambda)}$ via \texttt{fft2} for $p=8$}
\label{tab3}
\begin{tabular}{ccccccccccc}
\toprule
\multirow{2}{*}{$N_f$} &
\multicolumn{2}{c}{$\alpha=0.4$} &
\multicolumn{2}{c}{$\alpha=0.8$} &
\multicolumn{2}{c}{$\alpha=1.2$}&
\multicolumn{2}{c}{$\alpha=1.8$}\\
\cmidrule(lr){2-3} \cmidrule(lr){4-5} \cmidrule(lr){6-7} \cmidrule(lr){8-9}
 & $e(N_f)$ & rate & $e(N_f)$ & rate & $e(N_f)$ & rate& $e(N_f)$ &rate \\
\midrule
$2^6$ & 6.46e-05  & * & 8.94e-06   & * &  2.31e-06   & * & 3.69e-07   & * \\
$2^7$ &  8.46e-06  &   2.93  &    9.34e-07   & 3.26 & 1.90e-07 &  3.60   &    2.10e-08     &  4.14  \\
$2^8$ & 6.06e-07   &    3.80   & 5.23e-08  & 4.16  &   8.28e-09&   4.52   &   6.23e-10&  5.07\\
$2^9$ & 1.52e-08 &  5.32      &  1.02e-09  &    5.68 &  1.25e-10& 6.05  &   6.41e-12&  6.60 \\
$2^{10}$ &  4.80e-11      &  8.30  & 2.49e-12  &  8.68    &  2.36e-13 &  9.05&  8.15e-15  & 9.62  \\
\bottomrule
\end{tabular}
\end{table}
\subsection{The approximation accuracy for the TFL $(-\Delta)_{\lambda}^{\frac{\alpha}{2}}$}

To verify the error bound of the TFL as established in \cref{T2.1},
\cref{exa1} employs \cref{alg:1} (with a Fourier discretization resolution of $N_f=2^{14}$) to compute $(-\Delta_h)_{\lambda}^{\frac{\alpha}{2}}U$ for various spatial step sizes $h$ in the $d=2$ case. Given that an analytical expression for the TFL $(-\Delta)_{\lambda}^{\frac{\alpha}{2}}$ is generally unavailable, we evaluate the errors $e_{\infty}(h)$ and $e_{l^2}(h)$ to investigate their corresponding convergence rates
$\log_2\left(\frac{e_{\infty}(2h)}{e_{\infty}(h)}\right)$ and $\log_2\left(\frac{e_{l^2}(2h)}{e_{l^2}(h)}\right)$, where
 $$e_{\infty}(h)=\left\|(-\Delta_h)_{\lambda}^{\frac{\alpha}{2}}U-\left(-\Delta_{h/2}\right)_{\lambda}^{\frac{\alpha}{2}}U\right\|_{l^{\infty}(\Omega)},\ \ e_{l^2}(h)=\left\|(-\Delta_h)_{\lambda}^{\frac{\alpha}{2}}U-\left(-\Delta_{h/2}\right)_{\lambda}^{\frac{\alpha}{2}}U\right\|_{l^{2}(\Omega)}.$$
\begin{example}\label{exa1}
Consider the function  $$
u(x_1,x_2)=\left[\left(1-x_1^2\right)_{+}\left(1-x_2^2\right)_{+}\right]^{s}, \ \ (x_1,x_2)\in \mathbb{R}^2,\ \ s>0,$$ where $a_+=\max\{a,0\}$.
\end{example}
 As shown in \cref{tab4,tab5,tab6,tab7}, for $u\in \mathcal{B}^{s}(\mathbb{R}^2)$, the HFD schemes of orders $p=4, 6, 8$ for the TFL achieve the convergence rates of $O\left(h^{\min\{s-\alpha, p\}}\right)$ in the sense of $l^{\infty}$-norm error $e_{\infty}(h)$. These results correspond to the fixed tempering parameter $\lambda=0.5$, the regularity parameter $s=2, 3, 3.6, 6, 8, 10$, and the fractional power $\alpha=0.4, 1.8$.

For $s=5+\alpha, 2+\alpha$, the convergence rates in \cref{fig:1,fig:2} are $O\left(h^{\min\{s-\alpha, p\}}\right)$ for the $l^{\infty}$-norm error $e_{\infty}(h)$ and $O\left(h^{\min\{s-\alpha+1/2,\  p\}}\right)$ for the $l^2$-norm error $e_{l^2}(h)$. As shown, the HFD schemes with $p=4, 6, 8$ yield markedly lower errors than the $p=2$ benchmark scheme \cite{CDH2025}. We note that the dependence of the error on the order $p$ varies with $s$.  Specifically, for the higher-regularity case (e.g., $s=5+\alpha$), the error decreases as $p$ increases from 4 to 8. In contrast, for the moderate-regularity case (e.g., $s=2+\alpha$), the error increases as $p$ increases from 4 to 8.
Overall, the simulation results are consistent with \cref{T2.1}.

To visualize $(-\Delta)_{\lambda}^{\frac{\alpha}{2}}u(x_1,x_2)$, we compute $(-\Delta_h)_{\lambda}^{\frac{\alpha}{2}}U$ with $U=\{u_j\}_{x_j\in\Omega_h}$ as the reference solution using a fine grid $h=2^{-9}$. \cref{fig:3} displays the grid function plots of the TFL for $s=2+\alpha$ and $s=5+\alpha$ with $\alpha= 0.4$. These plots illustrate the nonlocal property of the TFL by showing that $(-\Delta)_{\lambda}^{\frac{\alpha}{2}}u(x_1,x_2)$ is nonzero even though $u(x_1,x_2)\equiv0$ for $(x_1,x_2)\in \mathbb{R}^2\backslash\Omega$.
  \begin{table}[H]
  \renewcommand{\arraystretch}{1.15} 
    \setlength{\tabcolsep}{8pt}
\centering
\begin{threeparttable}
\caption{Numerical accuracy in \cref{exa1}\ ($\alpha=0.4,\ s=2, 3, 3.6$).}
\label{tab4}
\begin{tabular}{cccccccc}
\toprule
\multirow{2}{*}{$h$} &
\multicolumn{2}{c}{$p=4,\ s=2$} &
\multicolumn{2}{c}{$p=6,\ s=3$} &
\multicolumn{2}{c}{$p=8,\ s=3.6$}\\
\cmidrule(lr){2-3} \cmidrule(lr){4-5} \cmidrule(lr){6-7}
 & $e_{{\infty}}(h)$ & rate & $e_{{\infty}}(h)$ & rate & $e_{{\infty}}(h)$ & rate \\
\midrule
$2^{-5}$ &  9.42e-05  & * &  8.41e-06 & * &   2.36e-06  & * \\
$2^{-6}$ & 3.07e-05 & 1.62 & 1.34e-06    & 2.65 &  2.46e-07 & 3.26   \\
$2^{-7}$ & 1.01e-05 &   1.61  & 2.17e-07  &  2.62  & 2.62e-08   &  3.23 \\
$2^{-8}$ &  3.31e-06 & 1.60   & 3.55e-08 &  2.61
 &  2.84e-09 & 3.21  \\
\bottomrule
\end{tabular}
\end{threeparttable}
\end{table}
\begin{table}[H]
\renewcommand{\arraystretch}{1.15} 
    \setlength{\tabcolsep}{8pt}
\centering
\begin{threeparttable}
\caption{Numerical accuracy in \cref{exa1}\ ($\alpha=1.8,\ s=2, 3, 3.6$).}
\label{tab5}
\begin{tabular}{cccccccc}
\toprule
\multirow{2}{*}{$h$} &
\multicolumn{2}{c}{$p=4,\ s=2$} &
\multicolumn{2}{c}{$p=6,\ s=3$} &
\multicolumn{2}{c}{$p=8,\ s=3.6$}\\
\cmidrule(lr){2-3} \cmidrule(lr){4-5} \cmidrule(lr){6-7}
 & $e_{{\infty}}(h)$ & rate & $e_{{\infty}}(h)$ & rate & $e_{{\infty}}(h)$ & rate \\
\midrule
$2^{-5}$ &   3.69e-01  & * &  1.89e-02  & * &   5.65e-03    & * \\
$2^{-6}$ &  3.16e-01   &  0.22 & 8.31e-03     &  1.19  &  1.59e-03  & 1.83   \\
$2^{-7}$ & 2.73e-01 &   0.21   & 3.63e-03 &   1.19  & 4.50e-04    &   1.82   \\
$2^{-8}$ &  2.37e-01  &  0.21   & 1.58e-03 &  1.20
 & 1.28e-04 &1.81  \\
\bottomrule
\end{tabular}
\end{threeparttable}
\end{table}
\begin{table}[H]
\renewcommand{\arraystretch}{1.15} 
    \setlength{\tabcolsep}{8pt}
\centering
\begin{threeparttable}
\caption{Numerical accuracy in \cref{exa1}\ ($\alpha=0.4,\ s=6, 8, 10$).}
\label{tab6}
\begin{tabular}{cccccccc}
\toprule
\multirow{2}{*}{$h$} &
\multicolumn{2}{c}{$p=4,\ s=6$} &
\multicolumn{2}{c}{$p=6,\ s=8$} &
\multicolumn{2}{c}{$p=8,\ s=10$}\\
\cmidrule(lr){2-3} \cmidrule(lr){4-5} \cmidrule(lr){6-7}
 & $e_{{\infty}}(h)$ & rate & $e_{{\infty}}(h)$ & rate & $e_{{\infty}}(h)$ & rate \\
\midrule
$2^{-3}$ &   2.23e-03  & * &   6.98e-04   & * &    3.77e-04    & * \\
$2^{-4}$ & 1.46e-04  &   3.94 & 1.24e-05     &   5.81        &  1.97e-06  &  7.58   \\
$2^{-5}$ & 9.21e-06 &  3.98   & 2.01e-07   &    5.95    &  8.28e-09    &  7.89   \\
$2^{-6}$ &  5.77e-07  & 4.00   &  3.16e-09 &  5.99
 &  3.28e-11  & 7.98  \\
\bottomrule
\end{tabular}
\end{threeparttable}
\end{table}
\begin{table}[H]
\renewcommand{\arraystretch}{1.15} 
    \setlength{\tabcolsep}{8pt}
\centering
\begin{threeparttable}
\caption{Numerical accuracy in \cref{exa1}\ ($\alpha=1.8,\ s=6, 8, 10$).}
\label{tab7}
\begin{tabular}{cccccccc}
\toprule
\multirow{2}{*}{$h$} &
\multicolumn{2}{c}{$p=4,\ s=6$} &
\multicolumn{2}{c}{$p=6,\ s=8$} &
\multicolumn{2}{c}{$p=8,\ s=10$}\\
\cmidrule(lr){2-3} \cmidrule(lr){4-5} \cmidrule(lr){6-7}
 & $e_{{\infty}}(h)$ & rate & $e_{{\infty}}(h)$ &rate & $e_{{\infty}}(h)$ & rate\\
\midrule
$2^{-3}$ &    1.27e-01    & * &   5.57e-02    & * &    3.94e-02   & * \\
$2^{-4}$ & 8.29e-03 &   3.93   & 9.96e-04     &   5.81        &   2.07e-04 &    7.57         \\
$2^{-5}$ &  5.24e-04  &  3.98   & 1.61e-05   &    5.95    &   8.71e-07   &  7.89   \\
$2^{-6}$ &  3.29e-05  & 4.00   &  2.53e-07   &  5.99
 &   3.45e-09   & 7.98  \\
\bottomrule
\end{tabular}
\end{threeparttable}
\end{table}
\begin{figure}[H]
\centering
\includegraphics[width=0.39\linewidth]{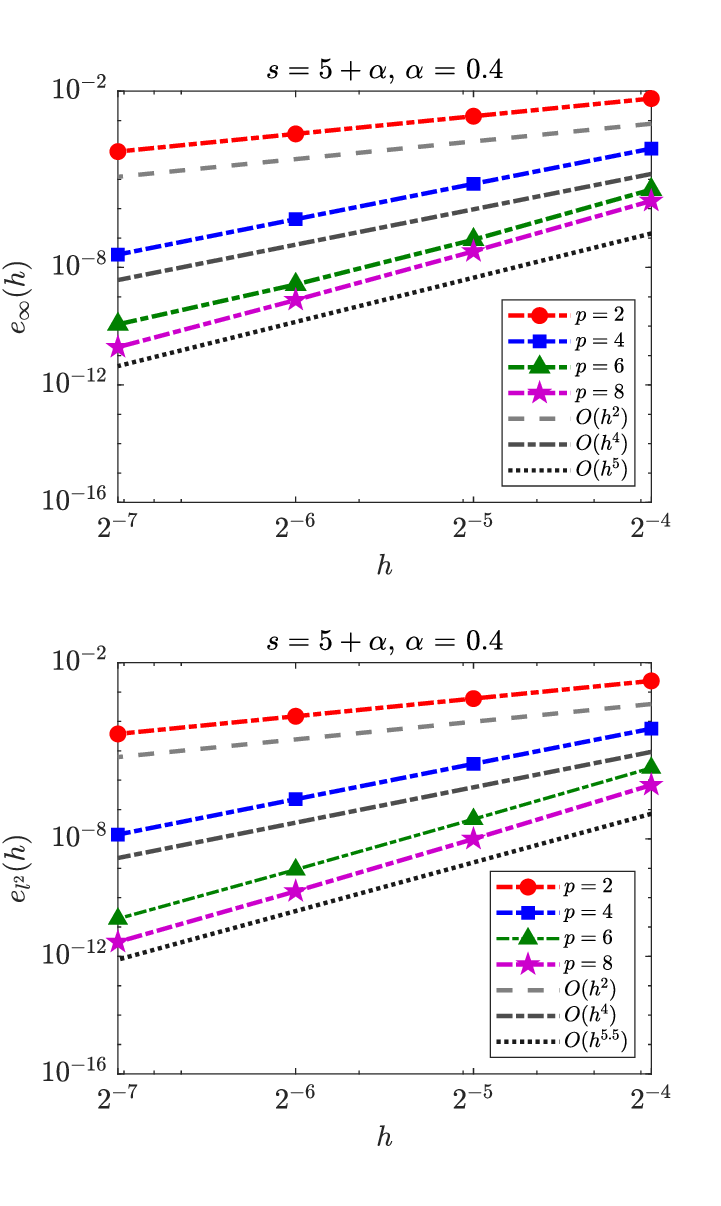}
\includegraphics[width=0.39\linewidth]{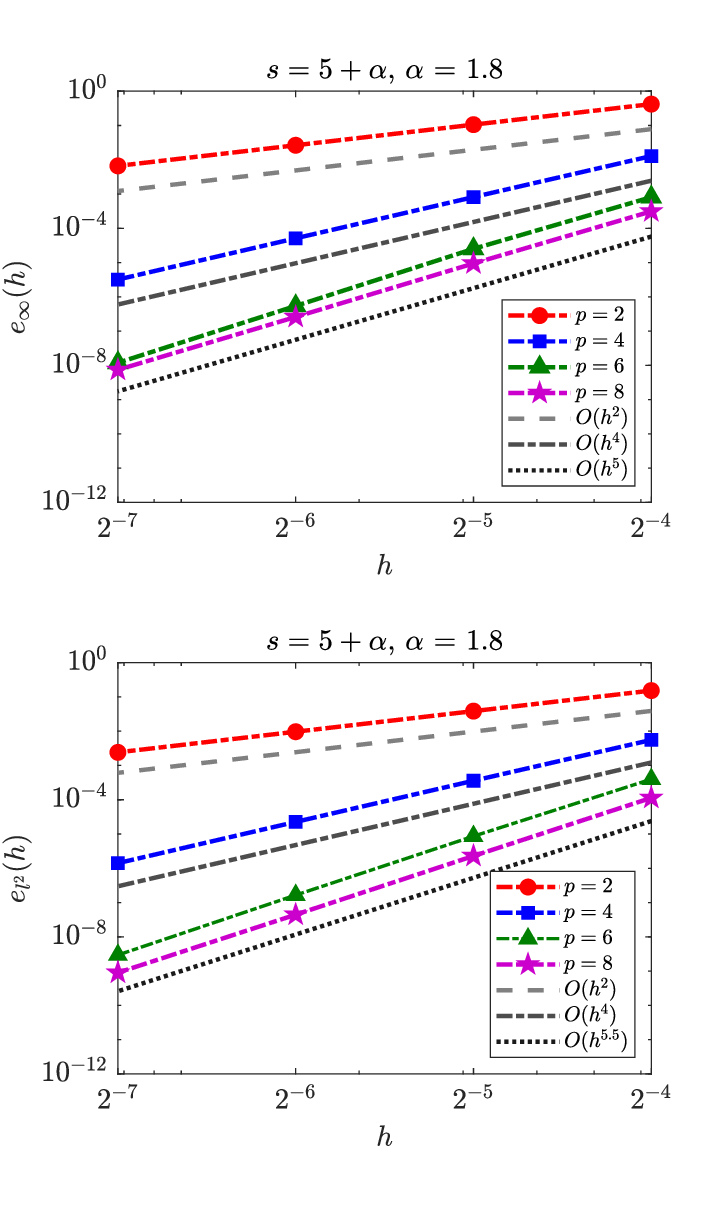}
\vspace{-0.88cm}
\caption{Numerical errors $e_{\infty}(h)$ and $e_{l^2}(h)$ of $(-\Delta_h)_{\lambda}^{\frac{\alpha}{2}}u(x_1,x_2)$ with $s=5+\alpha$ in \cref{exa1}.}
\label{fig:1}
\end{figure}
\vspace{-0.8cm}
\begin{figure}[H]
\centering
\includegraphics[width=0.39\linewidth]{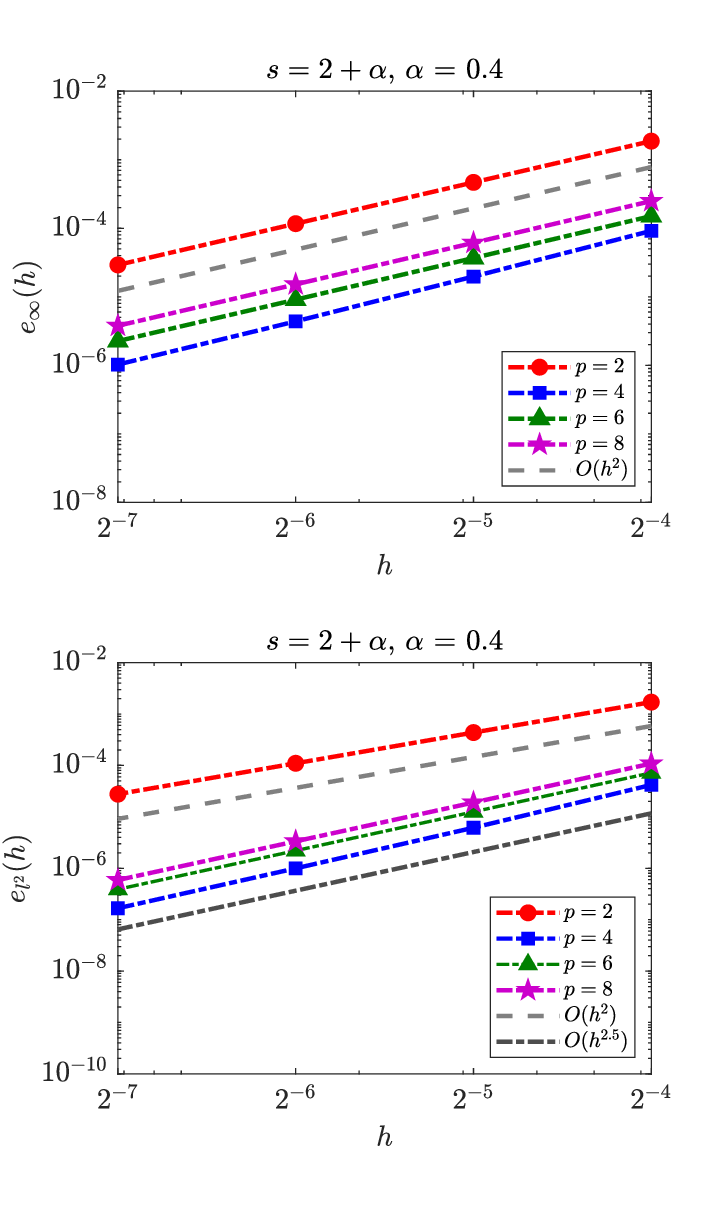}
\includegraphics[width=0.39\linewidth]{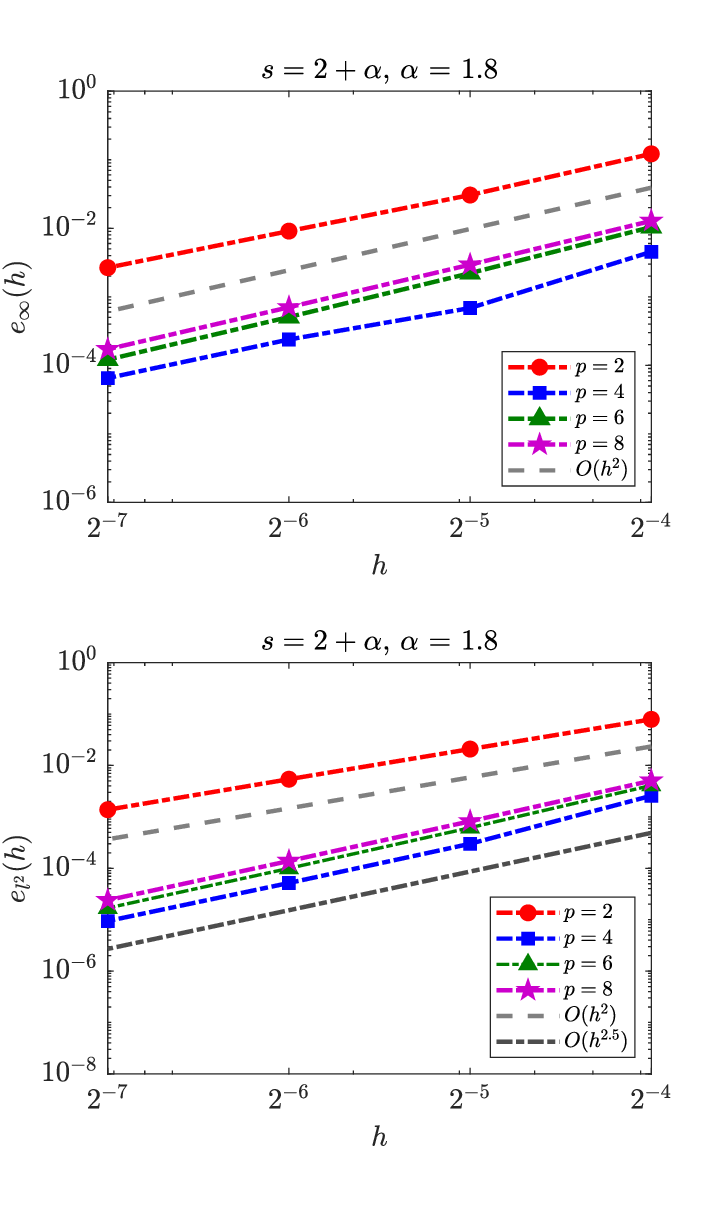}
\vspace{-0.88cm}
\caption{Numerical errors $e_{\infty}(h)$ and $e_{l^2}(h)$ of $(-\Delta_h)_{\lambda}^{\frac{\alpha}{2}}u(x_1,x_2)$ with $s=2+\alpha$ in \cref{exa1}.}
\label{fig:2}
\end{figure}

\begin{figure}[H]
\centering
\includegraphics[width=0.4\linewidth]{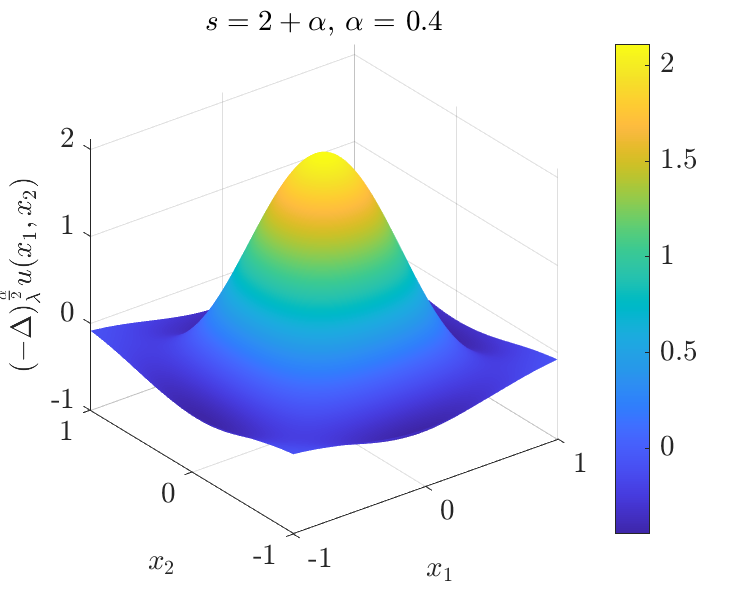}
\includegraphics[width=0.4\linewidth]{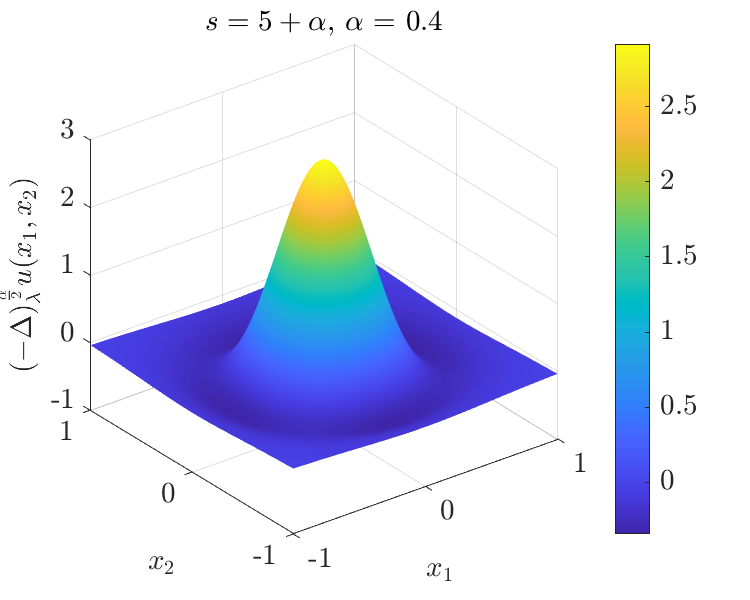}
\caption{TFL $(-\Delta)_{\lambda}^{\frac{\alpha}{2}}u(x_1,x_2)$ in \cref{exa1} ($\alpha=0.4,\ s=2+\alpha, 5+\alpha$).}
\label{fig:3}
\end{figure}
\subsection{The numerical approximations for the TFL equations}
In this subsection, we test the performance of the HFD methods with $p=4, 6, 8$ for the TFL equations. To compute $U^h$ in \cref{alg:2}, we use $N_f=2^{20}$ for $d=1$, $N_f=2^{14}$ for $d=2$, with a PCG tolerance of $10^{-16}$ and a zero initial guess. We measure the $l^\infty$-norm and $l^2$-norm errors between the numerical solution $U^h$ and the exact solution grid function $U$ to verify the theoretical results derived in \cref{d3.1}.
\begin{example}\label{exa2}(The TFL equation for $d=1$)\cite{DZ2019JSC}. For $x\in \mathbb{R}$ and $s>0$, the function $u(x)=\left(1-x^2\right)_{+}^s$ satisfies the following equation
\begin{align}\label{eq:5.1}
(-\Delta)_{\lambda}^{\frac{\alpha}{2}}u(x)&=f(x),\ \ x\in\Omega=(-1,1),\nonumber\\
u(x)&=0,\ \ x\in \Omega^c\equiv \mathbb{R}\backslash \Omega.
\end{align}
\end{example}
 In \cref{exa2}, the exact solution $u\in \mathcal{B}^s(\mathbb{R})$ is known, whereas the source term $f(x)$ is unknown. We approximate the grid function $f=\{f_j\}_{x_j\in\Omega_h}$ by computing $(-\Delta_h)_{\lambda}^{\frac{\alpha}{2}}U$ with $U=\{u_j\}_{x_j\in\Omega_h}$, using a fine spatial step size $h=2^{-12}$. The $l^\infty$-norm and $l^2$-norm numerical errors are defined as 
$E_{\infty}(h)=\left\|U-U^h\right\|_{l^{\infty}}$ and $E_{l^2}(h)=\left\|U-U^h\right\|_{l^{2}}.$ To illustrate the influence of the tempering parameter $\lambda$ on the numerical accuracy, \cref{fig:4}
 compares the $l^{\infty}$-norm errors for different $\lambda$, using $\alpha=0.4, 1.8$, $s=4+\alpha$, and the HFD method with $p=4$ for the TLF equation. This figure shows that the errors decrease as $\lambda$ decreases for small $\alpha$, but for large $\alpha$, they are insensitive to $\lambda$. As shown, the observed convergence rate is almost independent of both $\alpha$ and $\lambda$, and remains consistently of order $O(h^4)$.

In measuring the pointwise errors $\left|e_u^h\right|=\left|u(x)-u^h(x)\right|,\ x\in \Omega,$ with $\lambda=0.5,\ \alpha=0.4,\ s=6$ and $h=2^{-5}$, the left panel of \cref{fig:5} demonstrates that the scheme of $p=4$ yields significantly lower pointwise errors than the scheme of $p=2$ from Ref. \cite{CDH2025} at $P=3000$ interior nonmesh points. The numerical solution at these points is computed via \text{sinc} interpolation
$$I_Pu(y)=\sum\limits_{{x_j}\in \Omega_h}\text{sinc}\left(\frac{y-x_j}{h}\right)u^h_j=\sum\limits_{{x_j}\in \Omega_h}\frac{\sin\left(\pi(y-x_j)/{h}\right)}{\pi(y-x_j)/{h}}u^h_j,\ \ y\in\Omega.$$
The right panel presents the $l^2$-norm errors at $P=3000$ interior nonmesh points versus the spatial step size $h$ for $\lambda=0.5$, $\alpha=1.8$ and $s=6$. It is observed that for $s=6$ the schemes with $p=4, 6, 8$ achieve a convergence rate of $O\left(h^{\min\{s,p\}}\right)$, while the scheme with $p=2$ does not exceed second order. The $l^2$-norm error is computed as
$$E_{l^2}(h)=\left(\frac{1}{P}\sum\limits_{j=1}\limits^{P}{\left|u(x_j)-u^{h_P}(x_j)\right|^2}\right)^{1/2},\ \Omega_{h_P}=\left\{x_j\ |\ \ x_j = -1+jh_P,\ j=1,\ldots,P,\ h_P=\frac{2}{P+1}\right\}.$$
As shown, for the high-regularity solutions $\left(e.g., u\in \mathcal{B}^6(\mathbb{R})\right)$, the  schemes with $p=4, 6, 8$ clearly outperform the scheme with $p=2$ in Ref. \cite{CDH2025}. For the lower-regularity solutions $\left(e.g., u\in \mathcal{B}^1(\mathbb{R})\right)$, \cref{fig:6} reveals that the numerical errors of the scheme with $p=2$ are slightly lower than those of the schemes with $p=4, 6, 8$ for $\lambda=0.5$ and $\alpha=0.4, 1.8$.
\begin{figure}[H]
\centering
\includegraphics[width=0.43\linewidth]{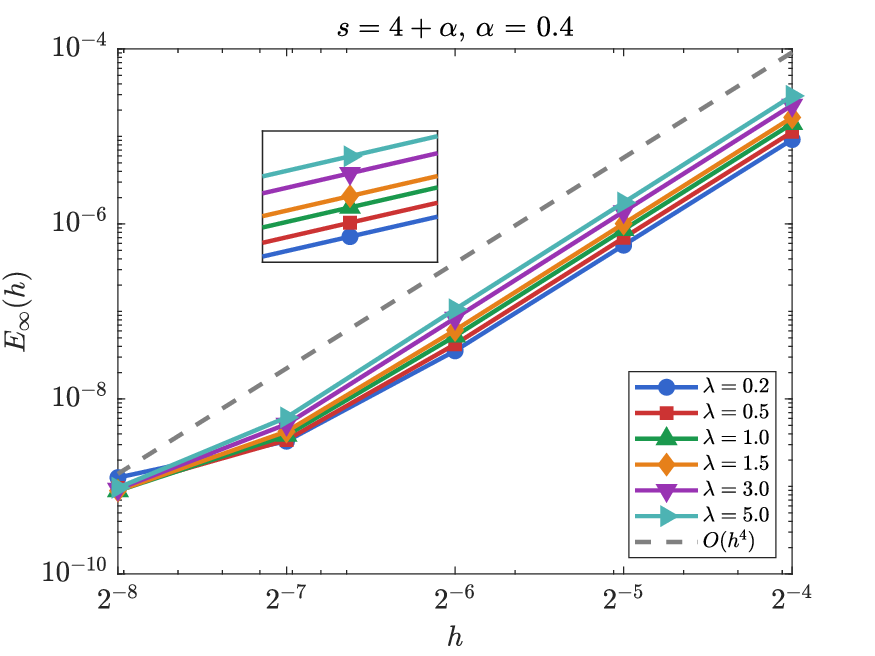}
\includegraphics[width=0.43\linewidth]{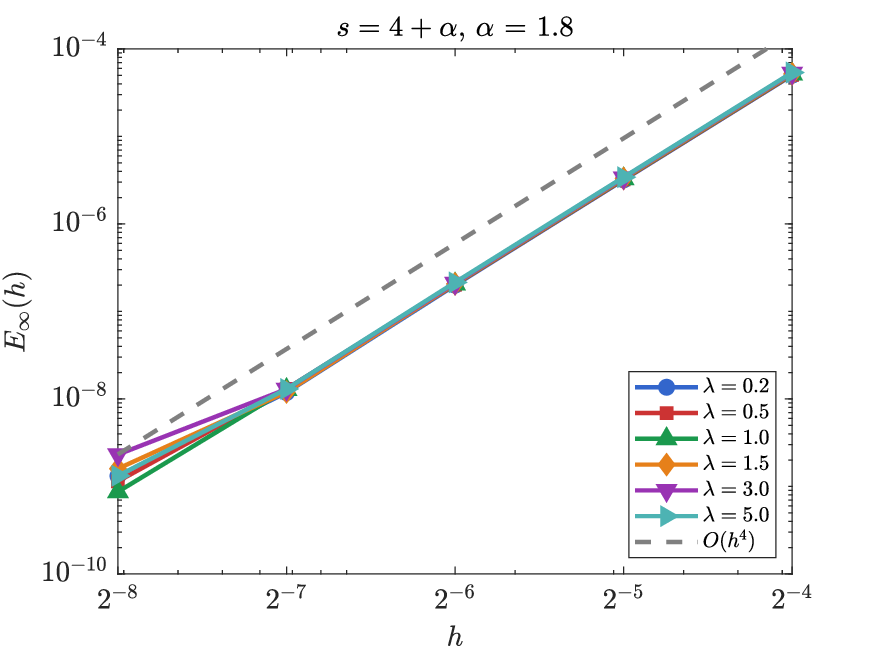}
\caption{Numerical error $E_{\infty}(h)$ in \cref{exa2} ($\lambda= 0.2, 0.5, 1.0, 1.5, 3.0, 5.0,\ \alpha=0.4, 1.8).$}
\label{fig:4}
\end{figure}
\begin{figure}[H]
\centering
\includegraphics[width=0.43\linewidth]{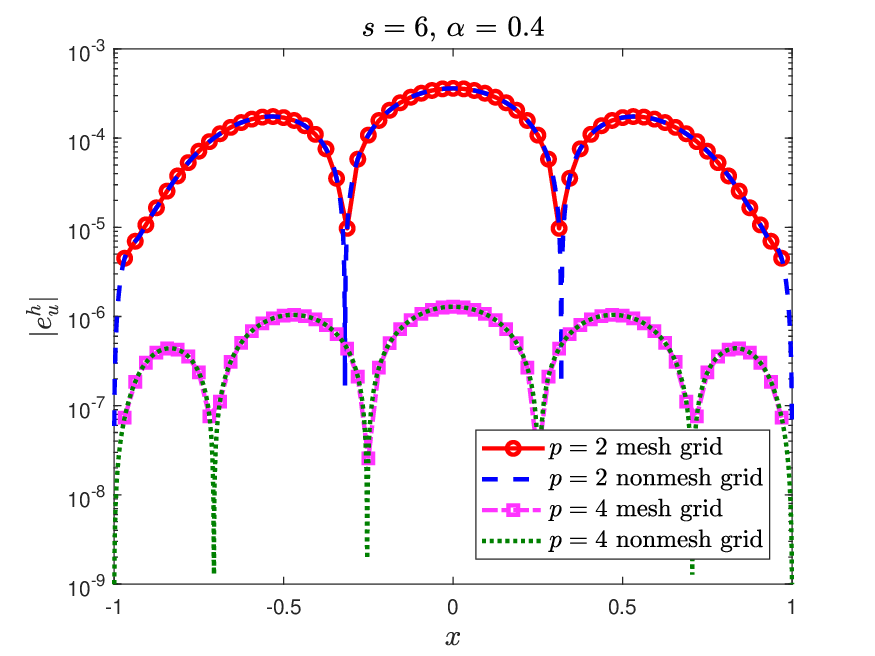}
\includegraphics[width=0.43\linewidth]{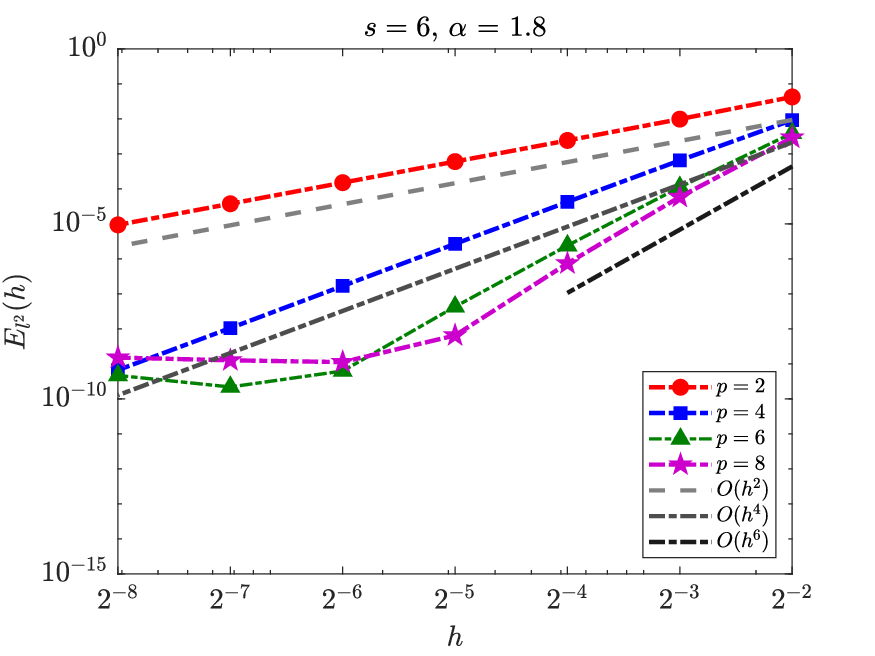}
\caption{Numerical accuracy of nonmesh points in \cref{exa2} ($\lambda = 0.5,\ \alpha = 0.4,\ 1.8,\ s = 6$).}
\label{fig:5}
\end{figure}
\begin{figure}[H]
\centering
\includegraphics[width=0.43\linewidth]{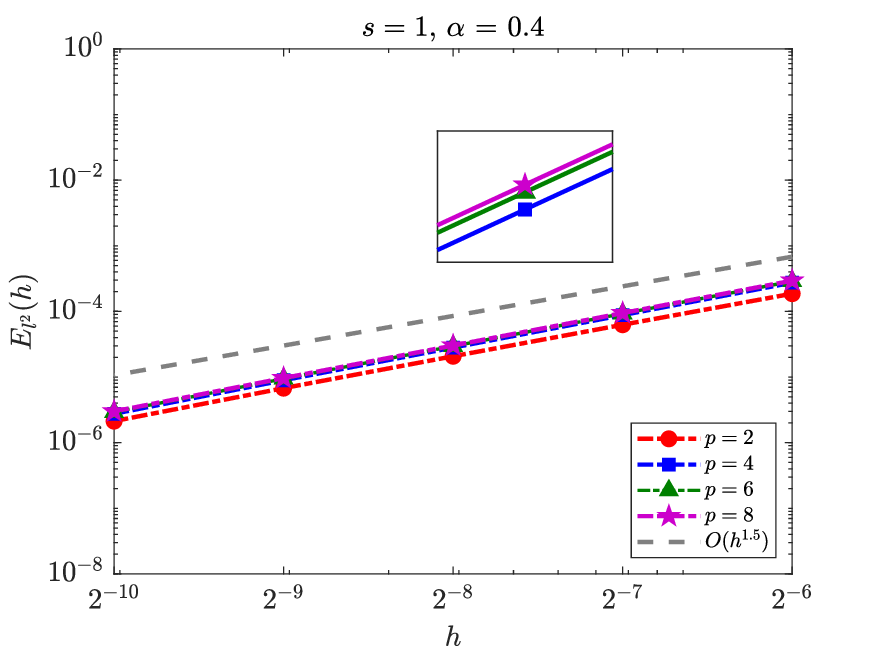}
\includegraphics[width=0.43\linewidth]{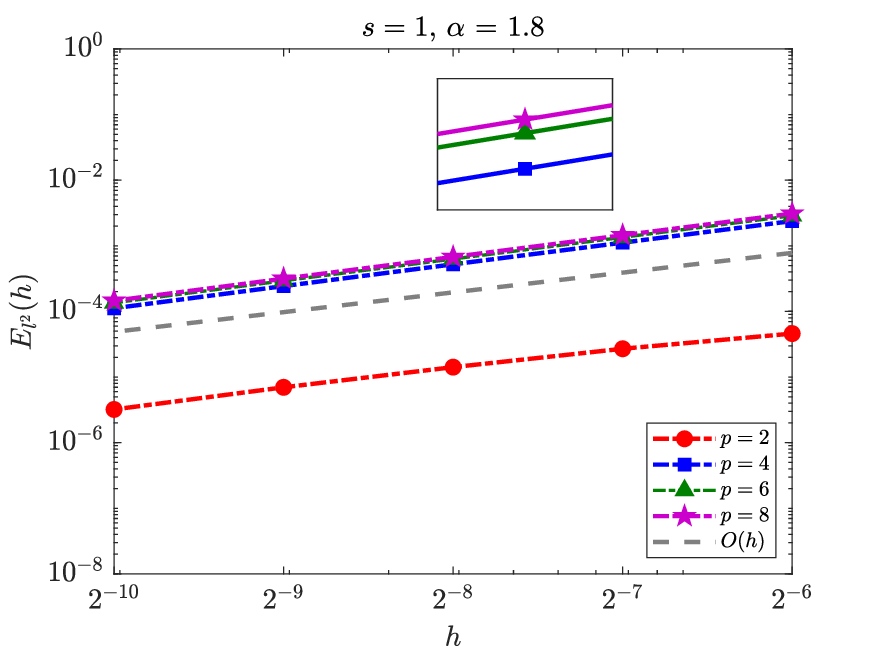}
\caption{Numerical accuracy in \cref{exa2} ($\lambda=0.5,\ \alpha=0.4, 1.8,\ s=1$).}
\label{fig:6}
\end{figure}
\begin{example}\label{exa3}(The TFL equations for $d=2$)\cite{CDH2025}. We solve the stochastic differential equations driven by $\alpha$-stable L$\acute{\text{e}}$vy noise
\begin{align}\label{eq:5.1}
(-\Delta)_{\lambda}^{\frac{\alpha}{2}}u(x)-\sigma\Delta u(x)+\nu u(x)&=f(x),\ \ x\in\Omega=[0,1]^2,\nonumber\\
u(x)&=0,\ \ x\in \Omega^c\equiv \mathbb{R}^2\backslash \Omega.
\end{align}
\end{example}
For Case 1, the exact solution is given by $u(x)=\left[\left(1-x_1^2\right)_{+}\left(1-x_2^2\right)_{+}\right]^{s}$, where $x=(x_1,x_2)\in \mathbb{R}^2$, $s>0$. The corresponding source term $f(x)$ lacks an explicit form. Using a fine spatial step size $h=2^{-9}$, we approximate the grid function $f=\{f_j\}_{x_j\in\Omega_h}$ by $\tilde{f}=(-\Delta_h)_{\lambda}^{\frac{\alpha}{2}}U-\sigma\Delta_hU+\nu U$, where $U=\{u_j\}_{x_j\in\Omega_h}$. Analogous to the $l^2$-norm error in \cref{exa1}, the $l^2$-norm error of the numerical solution is defined as $e_{l^2}(h)=\left\|U^h-U^{h/2}\right\|_{l^{2}}$.
With $\sigma = 0,\ \nu = 1$, $\lambda=0.5$ and $u\in\mathcal{B}^{s}(\mathbb{R}^2)$ ($s=2, 6$), \cref{tab8} illustrates that the scheme with $p=4$ for the TFL equation yields higher than expected accuracy for small $\alpha$. For large $\alpha$, the convergence rate is $O\left(h^{\min\{s,p\}}\right)$. Specifically,
 when $s<p$, the observed convergence rate is higher than the theoretical estimate $O\left(h^{\min\{s-\alpha,p\}}\right)$ derived in \cref{d3.1};
when $s>p$, the $l^2$-norm error converges at the optimal rate $O(h^p)$, which is consistent with \cref{d3.1}.
Analogous results for the scheme with $p=6$ and the scheme with $p=8$ are presented in \cref{tab9} and \cref{tab10}, respectively.

For Case 2, where $\sigma = 0$, $\nu =0$ and the source term $f(x)=1$, the exact solution is unknown. \cref{tab11} shows the results of the scheme with $p=4$ for the TFL equation.
 For small $\alpha$, both the $l^2$-norm errors and convergence rates decrease slightly as $\lambda$ decreases, while their differences are negligible.
For large $\alpha$, these errors and convergence rates are nearly identical across different $\lambda$, and the observed rates follow $O\left(h^{\min\{1,(\alpha+1)/2\}}\right)$.
     Overall, the dependence of the numerical error on both $\lambda$ and $\alpha$ aligns with the trends observed in \cref{exa2}.

For Case 3, where $\sigma=1,\ \nu=0,\ \lambda=0.5$ and the source term $f(x)=1$, there is no analytical solution. Given the low regularity of the solution, we adopt a central difference scheme for $-\Delta_h$ provided in \eqref{eq:3.1}: \begin{align*}-\Delta_h u(x)=\sum\limits_{l=1}^d\frac{1}{h^2}(-u(x-he_l)+2u(x)-u(x+he_l)).
\end{align*} As shown in \cref{fig:7}, when $\alpha$ is small (e.g., $\alpha=0.4$), the classical Laplacian dominates the governing operator, leading to a higher-regularity solution. The methods with $p=2, 4, 6, 8$ for the TFL equation all achieve a second-order convergence rate, with $l^2$-norm errors of comparable magnitude. When $\alpha$ is large (e.g., $\alpha=1.8$), the influence of the TFL is significantly enhanced, resulting in a low-regularity solution near the boundary and a convergence rate of $O(h)$. In conclusion, for high-regularity solutions, the schemes with $p=4, 6, 8$ for the TFL equations perform significantly better than the scheme with $p=2$. In contrast, for low-regularity solutions, the schemes with $p=4, 6, 8$ show no advantage over the scheme $p=2$  in terms of errors, and their errors are in fact comparable or even slightly larger.
\begin{table}[H]
\centering
\renewcommand{\arraystretch}{1.15} 
    \setlength{\tabcolsep}{8pt}
\caption{Numerical accuracy $O\left(h^{\min\{s, p\}}\right)$ of $e_{l^{2}}(h)$ in \cref{exa3}  ($p=4,\ \alpha=0.4, 1.8,\ s=2, 6$).}
\label{tab8}
\begin{tabular}{cccccccccc}
\toprule
\multirow{2}{*}{$h$} &
\multicolumn{2}{c}{$s=2,\ \alpha=0.4$} &
\multicolumn{2}{c}{$s=2,\ \alpha=1.8$} &
\multicolumn{2}{c}{$s=6,\ \alpha=0.4$}&
\multicolumn{2}{c}{$s=6,\ \alpha=1.8$}\\
\cmidrule(lr){2-3} \cmidrule(lr){4-5} \cmidrule(lr){6-7} \cmidrule(lr){8-9}
 & $e_{l^{2}}(h)$ & rate & $e_{l^{2}}(h)$ & rate & $e_{l^{2}}(h)$ & rate & $e_{l^{2}}(h)$ & rate \\
\midrule
$2^{-4}$ &   2.48e-05    &  *  &   8.27e-04  &  *&     1.16e-05  &  *  &     3.58e-05   & * \\
$2^{-5}$ &   3.72e-06      &  2.74 &   1.89e-04   &   2.13 &     7.33e-07   &  3.98  &    2.27e-06   &  3.98  \\
$2^{-6}$ &    5.91e-07     &  2.66 &     4.36e-05    &    2.12 &    4.60e-08    & 3.99   &    1.42e-07    &  3.99  \\
$2^{-7}$ &    9.71e-08   & 2.61 &      1.01e-05  & 2.11 &     2.88e-09  & 4.00  &     8.89e-09   & 4.00  \\
\bottomrule
\end{tabular}
\end{table}
\begin{table}[H]
\centering
\renewcommand{\arraystretch}{1.15} 
    \setlength{\tabcolsep}{8pt}
\caption{Numerical accuracy $O\left(h^{\min\{s, p\}}\right)$ of $e_{l^{2}}(h)$ in \cref{exa3}  ($p=6,\ \alpha=0.4, 1.8,\ s=3, 8$).}
\label{tab9}
\begin{tabular}{cccccccccc}
\toprule
\multirow{2}{*}{$h$} &
\multicolumn{2}{c}{$s=3,\ \alpha=0.4$} &
\multicolumn{2}{c}{$s=3,\ \alpha=1.8$} &
\multicolumn{2}{c}{$s=8,\ \alpha=0.4$}&
\multicolumn{2}{c}{$s=8,\ \alpha=1.8$}\\
\cmidrule(lr){2-3} \cmidrule(lr){4-5} \cmidrule(lr){6-7} \cmidrule(lr){8-9}
 & $e_{l^{2}}(h)$ & rate & $e_{l^{2}}(h)$ & rate & $e_{l^{2}}(h)$ & rate & $e_{l^{2}}(h)$ & rate \\
\midrule
$2^{-4}$ &    2.11e-06    &   *  &      4.10e-05   & * &       7.56e-07 &  *  &         2.41e-06   & *  \\
$2^{-5}$ &   1.70e-07    &    3.64 &     5.11e-06   &   3.00   &      1.23e-08    &   5.94   &     3.92e-08  &     5.94    \\
$2^{-6}$ &    1.43e-08      &  3.57   &    6.13e-07  &     3.06 &     1.94e-10  &  5.98  &    6.19e-10    &  5.98    \\
$2^{-7}$ &    1.22e-09   &  3.55   &    7.12e-08  & 3.11 &   3.06e-12   &    5.99     &  1.00e-11 &     5.95    \\
\bottomrule
\end{tabular}
\end{table}
\begin{table}[H]
\centering
\renewcommand{\arraystretch}{1.15} 
    \setlength{\tabcolsep}{8pt}
\caption{Numerical accuracy $O\left(h^{\min\{s, p\}}\right)$ of $e_{l^{2}}(h)$ in \cref{exa3}  ($p=8,\ \alpha=0.4, 1.8,\ s=3.6, 10$).}
\label{tab10}
\begin{tabular}{cccccccccc}
\toprule
\multirow{2}{*}{$h$} &
\multicolumn{2}{c}{$s=3.6,\ \alpha=0.4$} &
\multicolumn{2}{c}{$s=3.6,\ \alpha=1.8$} &
\multicolumn{2}{c}{$s=10,\ \alpha=0.4$}&
\multicolumn{2}{c}{$s=10,\ \alpha=1.8$}\\
\cmidrule(lr){2-3} \cmidrule(lr){4-5} \cmidrule(lr){6-7} \cmidrule(lr){8-9}
 & $e_{l^{2}}(h)$ & rate & $e_{l^{2}}(h)$ & rate & $e_{l^{2}}(h)$ & rate & $e_{l^{2}}(h)$ & rate \\
\midrule
$2^{-3}$ &   2.05e-05    & * &    3.08e-04   & * &     1.79e-05     & * &     5.83e-05    & * \\
$2^{-4}$ &    9.69e-07   &   4.40   &     2.32e-05    &    3.73  &       9.77e-08  &    7.52 &       3.19e-07  &   7.51    \\
$2^{-5}$ &   4.91e-08   &     4.30  &    1.75e-06 &   3.73  &    4.18e-10    &   7.87     &      1.36e-09    &  7.89  \\
$2^{-6}$ &     2.60e-09       &  4.24   &      1.33e-07  &   3.72&     1.68e-12  &     7.96  &   5.48e-12  &   7.96    \\
\bottomrule
\end{tabular}
\end{table}
\begin{table}[H]
\centering
\renewcommand{\arraystretch}{1.15} 
    \setlength{\tabcolsep}{8pt}
\caption{Numerical accuracy $O\left(h^{\min\{1, (\alpha+1)/2\}}\right)$ of $e_{l^2}(h)$ in \cref{exa3} ($p=4,\ \lambda=0.2, 0.5,\ \alpha=0.4, 1.8$).}
\label{tab11}
\begin{tabular}{cccccccccc}
\toprule
\multirow{2}{*}{$h$} &
\multicolumn{2}{c}{$\lambda=0.2,\ \alpha=0.4$} &
\multicolumn{2}{c}{$\lambda=0.2,\ \alpha=1.8$} &
\multicolumn{2}{c}{$\lambda=0.5,\ \alpha=0.4$}&
\multicolumn{2}{c}{$\lambda=0.5,\ \alpha=1.8$}\\
\cmidrule(lr){2-3} \cmidrule(lr){4-5} \cmidrule(lr){6-7} \cmidrule(lr){8-9}
 & $e_{l^{2}}(h)$ & rate & $e_{l^{2}}(h)$ & rate & $e_{l^{2}}(h)$ & rate & $e_{l^{2}}(h)$ & rate \\
\midrule
$2^{-5}$ &    1.18e-02  &       *  &     8.69e-04 &     *  &   2.00e-02     &    *     &       9.19e-04   &  * \\
$2^{-6}$ &     6.85e-03        &   0.79    &       4.38e-04  &  0.99 &     1.13e-02    &  0.83  &   4.65e-04&     0.98  \\
$2^{-7}$ &      3.96e-03         &   0.79    &       2.20e-04   &    0.99&      6.33e-03  &    0.83  &    2.34e-04  &    0.99     \\
$2^{-8}$ &      2.30e-03      &    0.78    &      1.11e-04  &   0.99  &      3.56e-03   &   0.83  &   1.18e-04  &  0.99     \\
\bottomrule
\end{tabular}
\end{table}
\begin{figure}[H]
\centering
\includegraphics[width=0.4\linewidth]{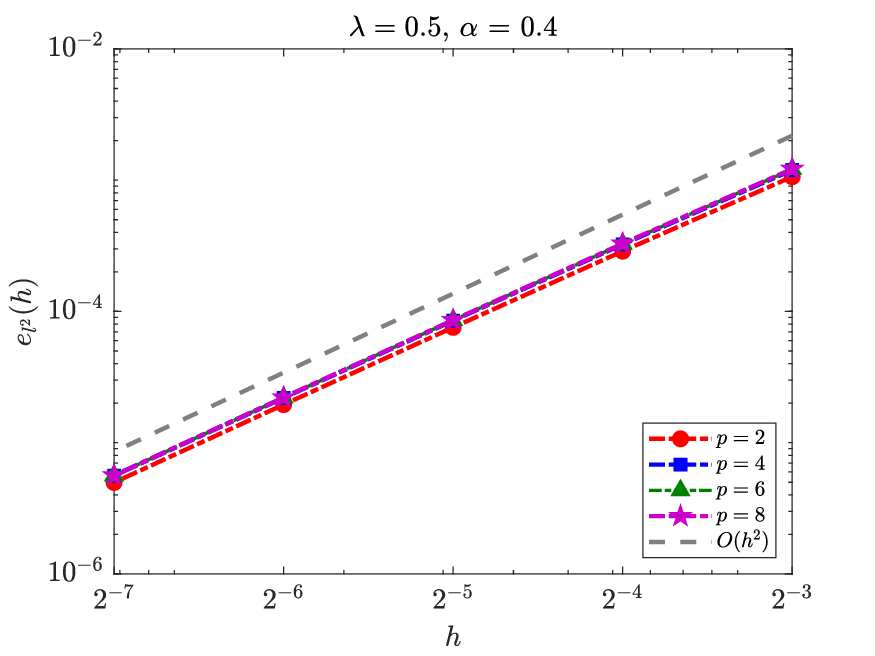}
\includegraphics[width=0.4\linewidth]{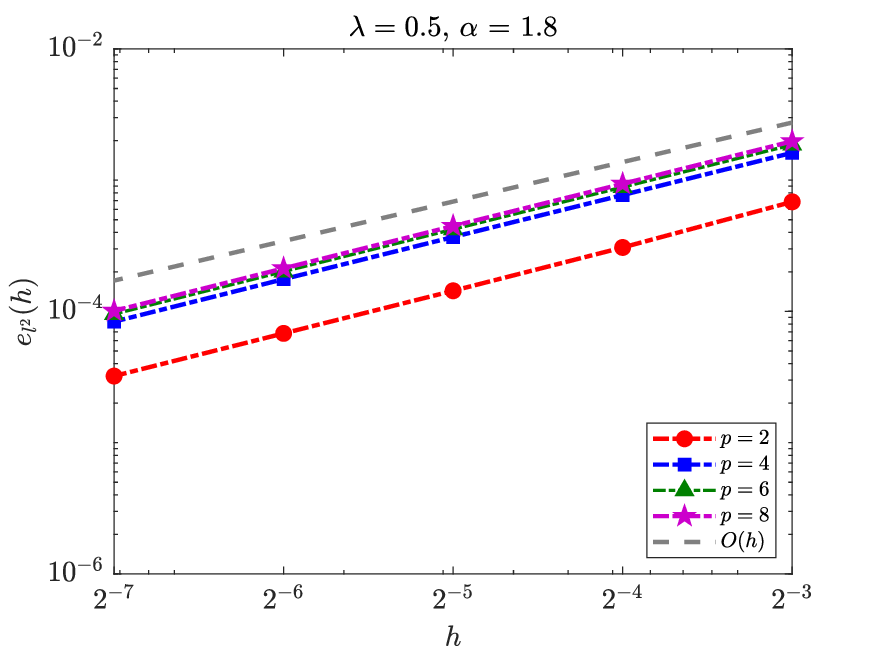}
\caption{Numerical accuracy of $e_{l^{2}}(h)$ in \cref{exa3} ($p =4, 6, 8,\ \lambda=0.5,\ \alpha = 0.4, 1.8$).}
\label{fig:7}
\end{figure}
\section{Conclusions}
In this work, we present a general framework of HFD schemes for the TFL, a pseudo-differential operator defined via the symbol $S^{\alpha}_{\lambda}(\xi)$ in \eqref{eq:4}. Based on the discrete symbol of the semi-discrete Fourier transform, we constructed new generating functions to obtain the TFL coefficients. The stability and convergence of the HFD methods for the TFL equations were analyzed, and the effectiveness of the proposed schemes was further verified through numerical examples.

The $l^{\infty}$-norm error of the HFD schemes for the TFL is of order $O\left(h^{\min\{s-\alpha,p\}}\right)$, with $s$ being the regularity of the exact solution and $p$ the order of the scheme.
 For the TFL equations, the HFD methods yield a convergence rate of $O\left(h^{\min\{s,p\}}\right)$ in the $l^{2}$-norm error for the high-regularity solutions, whereas for low-regularity ones, the convergence rate deteriorates to $O\left(h^{\min\{1,(\alpha+1)/{2}\}}\right)$. Additionally, our methods can also be applied to time-dependent problems and the approximation of anisotropic operators.
However, this work does not address the cases $\alpha=1$ and $\alpha=2$, nor does it provide complete theoretical estimates for the low-regularity functions, which will be the focuses of our future research.

\bibliographystyle{alpha}
\bibliography{TFD}
\end{document}